\documentclass{article}
\usepackage{amsmath,amsthm,amssymb,color,bm,fancybox,graphics}
\newcommand{\bigast}{\mathop{\vspace{-0.3em}\lower0.3ex\hbox{\scalebox{2}{$\ast$}}}}
\theoremstyle{plain}
\newtheorem{thm}{Theorem}
\newtheorem{prop}{Proposition}
\newtheorem{cor}{Corollary}
\newtheorem{lem}{Lemma}

\theoremstyle{definition}
\newtheorem{rem}{Remark}
\newtheorem{ass}{Assumption}
\title{On the value-distribution of the difference between logarithms of two symmetric power $L$-functions}
\author{Kohji Matsumoto\footnote{
The first author was supported by JSPS Grant-in-Aid for Scientific Research (B) Grant Number 25287002
}
and 
Yumiko Umegaki\footnote{
The second author was supported by JSPS Grant-in-Aid for Young Scientists (B) Grant Number 23740020.}
}
\date{}
\begin{document}
\maketitle
%
%
\begin{abstract}
We consider the value distribution of the difference 
between logarithms of two symmetric power $L$-functions 
at $s=\sigma > 1/2$.
We prove that certain averages of those values can be written as integrals 
involving a density function which is constructed explicitly.
\end{abstract}
%
%
\section{Introduction and the statement of main results.}

Let $f$ be a primitive form of weight $k$ and level $N$, 
which means that it is a normalized common Hecke eigen new form 
of weight $k$ for $\Gamma_0(N)$. 
We denote by $S_k(N)$ the set of all cusp forms of weight $k$ and level $N$. 
Any $f\in S_k(N)$ has a Fourier expansion at infinity of the form
\[
f(z)=\sum_{n=1}^{\infty}\lambda_f(n)n^{(k-1)/2}e^{2\pi inz},
\quad 
\lambda_f(1)=1.
\]
In the case $f$ is a normalized common Hecke eigen form, 
the Fourier coefficients $\lambda_f(n)$ are real numbers. 
We consider the $L$-function 
\[
L(f,s)
=
\sum_{n=1}^{\infty}\frac{\lambda_f(n)}{n^s} 
\]
associated with a primitive form $f$ where $s=\sigma+i\tau\in\mathbb{C}$.
This is absolutely convergent when $\sigma >1$, but can be continued to the whole of
$\mathbb{C}$ as an entire function.
 
We denote by $\mathbb{P}$ the set of all prime numbers.
We know that $L(f,s)$ has the Euler product
\begin{align*}
L(f,s)
=&
\prod_{\substack{p \in \mathbb{P} \\ p \mid N}}(1-\lambda_f(p) p^{-s})^{-1}
\prod_{\substack{p \in \mathbb{P} \\ p \nmid N}}(1-\lambda_f(p) p^{-s} + p^{-2s})^{-1}
\\
=&
\prod_{\substack{p \in \mathbb{P} \\ p \mid N}}(1-\lambda_f(p) p^{-s})^{-1}
\prod_{\substack{p \in \mathbb{P} \\ p\nmid N}}(1-\alpha_f(p) p^{-s})^{-1}(1-\beta_f(p) p^{-s})^{-1},
\end{align*}
where $\beta_f(p)$ is the complex conjugate of $\alpha_f(p)$.
Note that $\alpha_f(p)$ and $\beta_f(p)$ satisfy 
$\alpha_f(p)+\beta_f(p)=\lambda_f(p)$ and
$|\alpha_f(p)|=|\beta_f(p)|=1$.
This Euler product is deduced from the relations
\begin{equation}\label{euler}
\lambda_f(p^{\ell})=
\begin{cases}
\displaystyle \lambda_f^{\ell}(p) & p \mid N,\\
\displaystyle \sum_{h=0}^{\ell}\alpha_f^{\ell-h}(p)\beta_f^h(p) & p \nmid N.
\end{cases}
\end{equation}

In the present paper, we consider the value of
$\log L(\mathrm{Sym}_f^{\mu}, s)-\log L(\mathrm{Sym}_f^{\nu}, s)$ 
at $s=\sigma >1/2$,
where the $\gamma$-th symmetric power $L$-function is defined by 
\begin{align*}
L(\mathrm{Sym}_f^{\gamma}, s)
=&\sum_{n=1}^{\infty}\frac{\lambda_f(n^{\gamma})}{n^s}
\\
=
&
\prod_{\substack{p \in \mathbb{P}\\ p \mid N}}
(1-\lambda_f(p^{\gamma}) p^{-s})^{-1}
\prod_{\substack{p \in \mathbb{P} \\ p \nmid N}} \prod_{h=0}^{\gamma}
(1-\alpha_f^{\gamma-h}(p)\beta_f^h(p) p^{-s})^{-1}
\end{align*}
for $\sigma >1$. 
In the case $\gamma=1$, clearly
$L(\mathrm{Sym}_f^1, s)=L(f, s)$. 
In general, it is believed that the symmetric power $L$-function 
could be continued to an entire function and would satisfy a functional equation.
We suppose the analytic continuation and its holomorphy for $\sigma>1/2$ 
in Assumption~\ref{ass} below.

For $\sigma > 1$, we define
\[
\log L(\mathrm{Sym}_f^{\gamma}, s)
=
-
\sum_{\substack{p \in \mathbb{P}\\ p \mid N}}
\mathrm{Log} (1-\lambda_f(p^{\gamma}) p^{-s})
-
\sum_{\substack{p \in \mathbb{P} \\ p \nmid N}} 
\sum_{h=0}^{\gamma}
\mathrm{Log}(1-\alpha_f^{\gamma-h}(p)\beta_f^h(p) p^{-s}),
\]
where $\mathrm{Log}$ means the principal branch.
In the strip $1/2 < \sigma \leq 1$,
we suppose it can be analytically continued to $\sigma > 1/2$
under Assumption~\ref{grh} below, 
which claims that $L(\mathrm{Sym}_f^{\gamma}, s)$ 
has no zero in the strip $1/2 < \sigma \leq 1$.
In this paper we introduce the following two assumptions.
%
%
\begin{ass}\label{ass}
Let $f$ be a primitive form of even weight $k$ where $2\leq k < 12$ or $k=14$.
The level of $f$ is $q^m$, 
where $q$ is a prime number.
For a fixed positive integer $\gamma$, 
the symmetric power $L$-function $L(\mathrm{Sym}_f^{\gamma}, s)$ is
analytically continued to a holomorphic function in $\sigma >1/2$.
Moreover it satisfies the estimate
\begin{equation}\label{ass-estimate}
|L(\mathrm{Sym}_f^{\gamma}, s)|
\ll_{\gamma}
q^m(|\tau|+2)
\end{equation}
for $1/2 < \sigma \leq 2$.
\end{ass}

\begin{rem}
For the symmetric power $L$-function, 
if we obrain a suitable
functional equation which is the same type as in Cogdell and Michel~\cite{cm},
we have the same estimate as \eqref{ass-estimate} 
by using the Phlagm{\'e}n-Lindel{\"o}f principle.
As Cogdell and Michel mentioned in \cite{cm}, 
this assumption is held in the case when $f$ is 
a primitive form of weight $2$ and of square-free level
for the symmetric cube $L$-function,
which is proved by Kim and Shahidi \cite{ks-ann}. 
\end{rem}

\begin{rem}
For a primitive form of weight $k$ and level $M$, where $k$ is an
even positive integer and $M$ is a positive integer, the automophic $L$-function
$L(f,s)$ is entire and it has a functional equation.
The estimate of the form \eqref{ass-estimate} holds for $L(f,s)$,
that is
\begin{equation}\label{ass-estimate'}
|L(\mathrm{Sym}_f^1, s)|=|L(f, s)|
\ll
M (|\tau|+2)
\end{equation}
for $1/2 < \sigma \leq 2$.
\end{rem}

\begin{ass}\label{grh}
Let $f$ be a primitive form of weight $k$ which $2\leq k < 12$ or $k=14$.
The level is $q^m$, 
where $q$ is a prime number.
For a fixed positive integer $\gamma$, 
the $L$-functions $L(\mathrm{Sym}_f^{\gamma}, s)$
satisfies Generalized Riemann Hypothesis (GRH)
which means that $L(\mathrm{Sym}_f^{\gamma}, s)$ has no zero
in the strip $1/2 < \sigma \leq 1$.
\end{ass}

%
%
In this paper, we mainly consider two types of averages which are defined below.
For the definitions of them, we first prepare the notations.
Let $q$ be a prime number.
For any series $\{A_f\}$ over primitive forms 
$f \in S_k(q^m)$,
where $2\leq k < 12$ or $k=14$,
we use the symbol $\sum^{\prime}$ in the following sense:
\[
\sum_{f\in S_k(q^m)}^{\qquad\prime} A_f
=
\frac{1}{C_k(1-C_q(m))}
\sum_{\substack{f \in S_k(q^m)\\f:\mathrm{primitive}\;\mathrm{form}}}
\frac{A_f}{\langle f, f \rangle},
\]
where $C_k$ and $C_q(m)$ are constants defined by 
\[
C_k=\frac{(4\pi)^{k-1}}{\Gamma(k-1)},\quad
C_q(m)=\begin{cases}
       0 & m=1,\\
       q(q^2-1)^{-1} & m=2,\\
       q^{-1} & m\geq 3.
       \end{cases}
\]
These constants appeared in 
Lemma~3 in the second author~\cite{2ndAuthor} 
(see \eqref{P} below), 
which came from Petersson's formula.

We define the ``partial'' Euler product of the symmetric power  
$L$-function by 
\[
L_{\mathbb{P}(q)}(\mathrm{Sym}_f^{\gamma}, s)
=
\prod_{p \in \mathbb{P}(q)}\prod_{h=0}^{\gamma}
(1-\alpha_f^{\gamma -h}(p)\beta_f^h(p)p^{-s})^{-1}
\] 
for a primitive form $f$ of level $q^m$, 
where $q$ is a prime number
and the subset $\mathbb{P}(q) \subset \mathbb{P}$ means 
the set of all prime numbers except for 
the fixed prime number $q$.
Let $\mu >\nu \geq 1$ be integers with 
$\mu-\nu=2$.
By $Q(\mu)$ we denote the smallest prime number
satisfying $2^{\mu}/\sqrt{Q(\mu)}<1$.
In this paper, we study two types of averages which are defined by
\begin{align}\label{def_avg_prime}
&
\mathrm{Avg}_{\mathrm{prime}}
\Psi (
\log L_{\mathbb{P}(q)}(\mathrm{Sym}_f^{\mu},\sigma)-\log L_{\mathbb{P}(q)} (\mathrm{Sym}_f^{\nu}, \sigma)
)
\nonumber\\
=&
\lim_{\substack{q \to\infty\\ q :\mathrm{prime}\\ m:\mathrm{fixed}}}
\sum_{f\in S_k(q^m)}^{\qquad\prime}
\Psi (
\log L_{\mathbb{P}(q)}(\mathrm{Sym}_f^{\mu},\sigma)-\log L_{\mathbb{P}(q)}(\mathrm{Sym}_f^{\nu}, \sigma)
)
\end{align}
and
\begin{align}\label{def_avg_power}
&
\mathrm{Avg}_{\mathrm{power}}
\Psi (
\log L_{\mathbb{P}(q)}(\mathrm{Sym}_f^{\mu},\sigma)-\log L_{\mathbb{P}(q)}(\mathrm{Sym}_f^{\nu}, \sigma)
)
\nonumber\\
=&
\displaystyle
\lim_{\substack{m\to\infty\\ q :\mathrm{fixed} \; \mathrm{prime}\\ 
(\;\mathrm{If}\;1\geq \sigma > 1/2,\; q\geq Q(\mu))}}
\sum_{f\in S_k(q^m)}^{\qquad\prime}
\Psi (
\log L_{\mathbb{P}(q)}(\mathrm{Sym}_f^{\mu},\sigma)-\log L_{\mathbb{P}(q)} (\mathrm{Sym}_f^{\nu}, \sigma)), 
\end{align}
where $\Psi$ is a $\mathbb{C}$-valued function defined on $\mathbb{R}$.
On the above average $\mathrm{Avg}_{\mathrm{power}}$, 
we consider $q\geq Q(\mu)$ when $1\geq \sigma > 1/2$.
The reason is technical which will be mentioned in Section $5$.
The main theorem in the present paper is as follows. 
%
%
\begin{thm}\label{main}
Let $\mu > \nu  \geq 1$ be integers with $\mu-\nu=2$.
Suppose Assumptions~\ref{ass} and \ref{grh} 
when $\gamma$ is $\mu$ and $\nu$.
Let $k$ be an even integer which satisfies $2\leq k <12$ or $k=14$.
Then, for $\sigma>1/2$, there exists a function 
${\mathcal{M}}_{\sigma}:\mathbb{R}\to\mathbb{R}_{\geq 0}$ 
which can be explicitly constructed, and for which the formula
\begin{align*}
&
\mathrm{Avg}_{\mathrm{prime}} \Psi (
\log L_{\mathbb{P}(q)}(\mathrm{Sym}_f^{\mu},\sigma)-\log L_{\mathbb{P}(q)} (\mathrm{Sym}_f^{\nu}, \sigma)
)
\\
=&
\mathrm{Avg}_{\mathrm{power}} \Psi (
\log L_{\mathbb{P}(q)}(\mathrm{Sym}_f^{\mu},\sigma)-\log L_{\mathbb{P}(q)} (\mathrm{Sym}_f^{\nu}, \sigma)
)
\\
=&
\int_{\mathbb{R}}
\mathcal{M}_{\sigma}(u)
\Psi (u) \frac{du}{\sqrt{2\pi}}
\end{align*}
holds for any $\Psi:\mathbb{R}\to\mathbb{C}$ which is a
bounded continuous function or a compactly supported characteristic function.
\end{thm}
The above restriction on the weight $k$ is necessary to prove \eqref{P} below.

%
%
We mention a corollary of the $\mathrm{Avg}_{\mathrm{prime}}$ part of the theorem.
Consider the following different type of averages, involving summations with respect 
to levels:
\begin{align*}
&
\mathrm{Avg}_{\mathrm{primesum}}
\Psi (
\log L_{\mathbb{P}(q)}(\mathrm{Sym}_f^{\mu},\sigma)-\log L_{\mathbb{P}(q)} (\mathrm{Sym}_f^{\nu}, \sigma)
)
\\
=&
\lim_{X \to \infty}
\frac{1}{\pi(X)} \sum_{\substack{q \leq X\\ q :\mathrm{prime}\\ m : \mathrm{fixed}}} 
\sum_{f \in S_k(q^m)}^{\qquad\prime}
\Psi (
\log L_{\mathbb{P}(q)}(\mathrm{Sym}_f^{\mu},\sigma)-\log L_{\mathbb{P}(q)} (\mathrm{Sym}_f^{\nu}, \sigma)),
\end{align*}
where $\pi(X)$ denotes the number of prime numbers not larger than $X$, and
\begin{align*}
&
\mathrm{Avg}_{\mathrm{primepowersum}}
\Psi (
\log L_{\mathbb{P}(q)}(\mathrm{Sym}_f^{\mu},\sigma)-\log L_{\mathbb{P}(q)} (\mathrm{Sym}_f^{\nu}, \sigma)
)
\\
=&
\lim_{X \to \infty}
\frac{1}{\pi^*(X)} \sum_{\substack{q^m \leq X\\ q :\mathrm{prime}\\ m\geq 1}} 
\sum_{f \in S_k(q^m)}^{\qquad\prime}
\Psi (
\log L_{\mathbb{P}(q)}(\mathrm{Sym}_f^{\mu},\sigma)-\log L_{\mathbb{P}(q)} (\mathrm{Sym}_f^{\nu}, \sigma)),
\end{align*}
where $\pi^*(X)$ denotes the number of all pairs $(q,m)$ of a prime number $q$ and
a positive integer $m$ with $q^m\leq X$.

\begin{cor}\label{main-cor}
Under the same assumptions as Theorem~\ref{main}, we have
\begin{align*}
&
\mathrm{Avg}_{\mathrm{primesum}} \Psi (
\log L_{\mathbb{P}(q)}(\mathrm{Sym}_f^{\mu},\sigma)-\log L_{\mathbb{P}(q)} (\mathrm{Sym}_f^{\nu}, \sigma)
)
\\
=&
\mathrm{Avg}_{\mathrm{primepowersum}} \Psi (
\log L_{\mathbb{P}(q)}(\mathrm{Sym}_f^{\mu},\sigma)-\log L_{\mathbb{P}(q)} (\mathrm{Sym}_f^{\nu}, \sigma)
)
\\
=&
\int_{\mathbb{R}}
\mathcal{M}_{\sigma}(u)
\Psi (u) \frac{du}{\sqrt{2\pi}}.
\end{align*}
\end{cor}

\begin{rem}\label{rem-main}
Theorem~\ref{main} can be generalized to the case of any average defined by some
limit, which is different from those in \eqref{def_avg_prime} and
\eqref{def_avg_power}, but satisfies the condition that $q^m\to\infty$.
(For example, $q\to\infty$ with $m=m(q)$ moving arbitrarily.)    In fact, from the
proof we can see that the only necessary limit procedure is $q^m\to\infty$.
\end{rem}

The first result of this type is due to Bohr and Jessen \cite{bj}.
Let $\zeta(s)$ be the Riemann zeta-function.    Bohr and Jessen proved that,
when $\Psi$ is a compactly supported characteristic function defined on
$\mathbb{C}$, the formula
\begin{align}\label{bj-formula}
\lim_{T\to\infty}\frac{1}{2T}\int_{-T}^T \Psi(\log\zeta(s+i\tau))d\tau=
\int_{\mathbb{C}}\mathcal{M}_{\zeta,\sigma}(w)\Psi(w)\frac{dudv}{\sqrt{2\pi}}
\end{align}
holds for $\Re s>1/2$ (where $w=u+iv$), with a certain density function
$\mathcal{M}_{\zeta,\sigma}$.
The analogue for the logarithmic derivative $\zeta^{\prime}/\zeta(s)$
was first proved by Kershner and Wintner \cite{kewi}.

Ihara \cite{ihara} discovered that the same type of results can be shown for
certain mean values of $L^{\prime}/L(s,\chi)$ with respect to characters, where
$L(s,\chi)$ denotes the Dirichlet (or Hecke) $L$-function attached to the character
$\chi$, including also the function field case.
Ihara's work was strengthened, and extended to the $\log L$ case, in several joint papers
of Ihara and the first author 
\cite{i-m0}, \cite{i-1stAuthor}, \cite{i-m2}, \cite{i-m3}.
Recently, Mourtada and Murty \cite{mm} obtained an analogous result for the mean value
of $L^{\prime}/L(s,\chi)$ with respect to discriminants.

In those former results, the function $\Psi$ is defined on $\mathbb{C}$, and the
right-hand side of the formula is an integral over $\mathbb{C}$.   However in
our Theorem~\ref{main}, the function $\Psi$ is defined on $\mathbb{R}$, and the
right-hand side is an integral over $\mathbb{R}$.   This is one remarkable
difference of our presenr work from the former researches. 

The plan of this paper is as follows. 
Section 2 is the preparation, with the proof of Corollary~\ref{main-cor}.
In Section 3 we construct the density function ${\mathcal{M}}_{\sigma}$,
in Section 4 we state the key lemma (Lemma~\ref{keylemma}) 
and prove it in the case $\sigma > 1$,
in Section 5 we prepare certain approximation of 
$L_{\mathbb{P}(q)}(\mathrm{Sym}_f^{\gamma},s)$ 
to prove the key lemma, 
in Section 6 we prove the key lemma for $1\geq \sigma > 1/2$ and finally, 
in Section 7 we will complete the proof of Theorem~\ref{main}.
The basic structure of our argument 
is similar to the previous work by Ihara and the first author~\cite{i-1stAuthor}.

\begin{rem}
It is surely interesting to search for the density function, of the nature similar
to the above, for the average
of $\log L_{\mathbb{P}(q)}(\mathrm{Sym}_f^{\gamma}, s)$ itself
for $f\in S_k(q^m)$.
However to obtain such a density funcion is difficult by our present method.
The reason is explained in Remark~\ref{reason} below.
Therefore we consider the difference between the logarithm of
symmetric $\mu$-th power $L$ function
and that of symmetric $\nu$-th power $L$ function 
where $\mu$ and $\nu$ are of the same parity.
However we have another difficulty in the case $\mu-\nu>2$.
It is explained in Remark~\ref{nu} below.
Hence Theorem~\ref{main} is shown only in the case $\mu-\nu=2$.
\end{rem}

{\bf Acknowledgment.}
The authors would like to express their gratitude to Professor Masaaki Furusawa 
and Professor Atsuki Umegaki for their valuable comments.

%
%
\section{Preparations.}

Bohr and Jessen \cite{bj} used the Kronecker-Weyl theorem on uniform distribution of
sequences as an essential tool in the proof of \eqref{bj-formula}.   
In Ihara \cite{ihara}, the corresponding tool is the orthogonality relation of
characters.

In our present situation, the corresponding useful tool is Petersson's well-known
formula (see, e.g., \cite{iw}).
In the proof of our main Theorem~\ref{main}, we will use the following
formula (\eqref{P} below) for a prime number $q$, 
which was shown in Lemma~3 in the second author~\cite{2ndAuthor}.
This formula embodies the essence of Petersson's formula, in the form suitable for
our present aim.

When $2\leq k < 12$ or $k=14$,
for the primitive form $f$ of weight $k$ and level $q^m$, 
we have
\begin{equation}\label{P}
\sum_{f\in S_k(q^m)}^{\qquad\prime}
\lambda_f(n)
=
\delta_{1, n}
+
\begin{cases} 
O_k(n^{(k-1)/2} q^{-k+1/2})
& m=1, \\
O_k(n^{(k-1)/2} q^{m(-k+1/2)}q^{k-3/2})
& m\geq 2, 
\end{cases}
\end{equation}
where $\delta_{1, n}=1$ if $n=1$ and $0$ otherwise.
We denote the error term in \eqref{P} by
$n^{(k-1)/2} E(q^m)$, that is
\[
\sum_{f\in S_k(q^m)}^{\qquad\prime}
\lambda_f(n)
-
\delta_{1, n}
=
n^{(k-1)/2}E(q^m).
\]
Then we have
\begin{equation}\label{E1}
E(q^m)\ll q^{-k+1/2}
\end{equation}
for any $m$, 
and 
\begin{align}\label{E2}
E(q^m)\ll& \begin{cases}
          q^{-3/2} & m=1,\\
          q^{-5/2} & m=2,\\
          q^{-1-m}   & m\geq 3,
          \end{cases}
\nonumber \\
\ll& q^{-m}
\end{align}
for any $m$.
Also in the case $n=1$, the formula \eqref{P} implies
\begin{equation}\label{P1}
\sum_{f\in S_k(q^m)}^{\qquad\prime}
\lambda_f(1)
=
\sum_{f\in S_k(q^m)}^{\qquad\prime}
1
=
1+E(q^m) \ll 1.
\end{equation}

Let $\mathcal{P}$ be a subset of $\mathbb{P}$ and $q$ a fixed prime number.
For a primitive form $f$ of weight $k$ and level $q^m$, define
\[
L_{\mathcal{P}}(\mathrm{Sym}_f^{\gamma}, s)
=\prod_{p \in \mathcal{P}} L_p(\mathrm{Sym}_f^{\gamma},s),
\]
where
\[
L_p(\mathrm{Sym}_f^{\gamma},s)=
\begin{cases}
(1-\lambda_f(p^{\gamma})p^{-s})^{-1} & p=q,\\
\displaystyle 
\prod_{h=0}^{\gamma}(1-\alpha_f^{\gamma-h}(p)\beta_f^h(p)p^{-s})^{-1} & p\neq q.
\end{cases}
\]
Especially we write 
\[
L_{\mathcal{P}}(f,s)=L_{\mathcal{P}}(\mathrm{Sym}_f^1,s).
\]
Further, for integers $\mu>\nu>0$ with $\mu-\nu=2$,
we put
\[
L_{\mathcal{P}}(\mathrm{Sym}_f^{\mu}, \mathrm{Sym}_f^{\nu}, s)
=
\frac{L_{\mathcal{P}}(\mathrm{Sym}_f^{\mu},  s)}
{L_{\mathcal{P}}(\mathrm{Sym}_f^{\nu},  s)}.
\]
This can be defined for $\sigma>1/2$ under Assumption~\ref{grh}.

Now let $\mathcal{P}$ be a finite subset of $\mathbb{P}(q)$.
We define the topological group $\mathcal{T}_{\mathcal{P}}$ by
\[
\mathcal{T}_{\mathcal{P}}=\prod_{p \in \mathcal{P}} \mathcal{T},
\]
where $\mathcal{T}=\{t \in \mathbb{C} \mid |t|=1\}$.
For a fixed $\sigma>1/2$, we consider the function 
\[
g_{\sigma, \mathcal{P}}(t_{\mathcal{P}})
=
\sum_{p \in \mathcal{P}}
g_{\sigma, p}(t_p)
\]
on  $t_{\mathcal{P}}=(t_p)_{p \in \mathcal{P}} \in \mathcal{T}_{\mathcal{P}}$, where
\[
g_{\sigma, p}(t_p)=-\log(1-t_p p^{-\sigma}).
\]
For any $s=\sigma+i\tau$ $(\sigma >1/2)$ we have 
\begin{align}\label{g_sigma}
\lefteqn{
\log L_{\mathcal{P}}(\mathrm{Sym}_f^{\mu}, \mathrm{Sym}_f^{\nu}, s)}
\nonumber\\
&=
\log L_{\mathcal{P}}(\mathrm{Sym}_f^{\mu},  s)
-
\log L_{\mathcal{P}}(\mathrm{Sym}_f^{\nu},  s)
\nonumber\\
&=
\sum_{p \in \mathcal{P}} \Big(
-\sum_{h=0}^{\mu} 
\log(1-\alpha_f^{\mu-h}(p)\beta_f^h(p)p^{-s})
+
\sum_{h=0}^{\nu} 
\log(1-\alpha_f^{\nu-h}(p)\beta_f^h(p)p^{-s})
\Big)
\nonumber\\
&=
\sum_{p \in \mathcal{P}} \Big(
-\log(1-\alpha_f^{\mu}(p)p^{-s})
-
\log(1-\beta_f^{\mu}(p)p^{-s})
\Big)
\nonumber\\
&=
\sum_{p \in \mathcal{P}} 
\Big(
g_{\sigma, p}(\alpha_f^{\mu}(p) p^{-i\tau})
+g_{\sigma, p}(\beta_f^{\mu}(p) p^{-i\tau})
\Big)
\nonumber\\
&=
g_{\sigma, \mathcal{P}}(\alpha_f^{\mu}(\mathcal{P}) \mathcal{P}^{-i\tau})
+g_{\sigma, \mathcal{P}}(\beta_f^{\mu}(\mathcal{P}) \mathcal{P}^{-i\tau}),
\end{align}
where $\alpha_f^{\mu}(\mathcal{P})=(\alpha_f^{\mu}(p))_{p\in\mathcal{P}}$, 
$\beta_f^{\mu}(\mathcal{P})=(\beta_f^{\mu}(p))_{p\in\mathcal{P}}$ and
$\mathcal{P}^{-i\tau}=(p^{-i\tau})_{p\in\mathcal{P}}$.
In the above equation, we used the fact that
$\beta_f(p)$ is the complex conjugate of $\alpha_f(p)$, and hence
$\alpha_f^{\mu-h}(p)\beta_f^h(p)=\alpha_f^{\nu-(h-1)}(p)\beta_f^{h-1}(p)$ 
($1\leq h\leq \mu-1$).
Hereafter, we sometimes write $\alpha_f(p)=e^{i\theta_f(p)}$ and  
$\beta_f(p)=e^{-i\theta_f(p)}$.

In the case $\sigma>1$, 
we deal with the value  
$L_{\mathbb{P}(q)}(\mathrm{Sym}_f^{\gamma}, \sigma+i\tau)$
as the limit of the value  
$L_{\mathcal{P}}(\mathrm{Sym}_f^{\gamma}, \sigma+i\tau)$
as $\mathcal{P}$ tends to $\mathbb{P}(q)$.
In fact, from \eqref{g_sigma} we have 
\begin{align}\label{g_sigma2}
&
\log L_{\mathbb{P}(q)}(\mathrm{Sym}_f^{\mu}, \mathrm{Sym}_f^{\nu}, \sigma+i\tau)
\nonumber\\
&\qquad =
\lim_{\substack{\mathcal{P}\to\mathbb{P}(q)\\ \mathcal{P}\subset {\mathbb{P}(q)}}}
(
g_{\sigma, \mathcal{P}}(\alpha_f^{\mu}(\mathcal{P}) \mathcal{P}^{-i\tau})
+g_{\sigma, \mathcal{P}}(\beta_f^{\mu}(\mathcal{P}) \mathcal{P}^{-i\tau})
).
\end{align}

In the case $1\geq \sigma>1/2$, 
we will prove the relation between
$\log L_{\mathbb{P}(q)}(\mathrm{Sym}_f^{\gamma}, \sigma)$
and
$\log L_{\mathcal{P}}(\mathrm{Sym}_f^{\gamma}, \sigma)$
with a suitable finite subset $\mathcal{P} \subset \mathbb{P}(q)$ 
depending on $q^m$ and will consider the averages of them.
This will be given in Lemma~\ref{appSymL} in Section~5.

Now we conclude this section with the proof of Corollary~\ref{main-cor}.

\begin{proof}[Proof of Corollary~\ref{main-cor}]
Write
\[
A(q^m)=\sum_{f \in S_k(q^m)}^{\qquad\prime} \Psi (
\log L_{\mathbb{P}(q)}(\mathrm{Sym}_f^{\mu},\sigma)-\log L_{\mathbb{P}(q)} (\mathrm{Sym}_f^{\nu}, \sigma)).
\]
The $\mathrm{Avg}_{\mathrm{prime}}$ part of the theorem implies that, for any
$\varepsilon>0$, there exists a $Q_0=Q_0(\varepsilon)$ for each fixed $m$ such that
\[
\Biggl|A(q^m)
-\int_{\mathbb{R}}\mathcal{M}_{\sigma}(u)
\Psi (u) \frac{du}{\sqrt{2\pi}}\Biggr|<\varepsilon
\]
for any prime $q>Q_0$.    This clearly implies the first part of the corollary.
As for the second part, first note that $\pi^*(X)$ in the denominator can be
replaced by $\pi(X)$, because $\lim_{X\to\infty}\pi(X)/\pi^*(X)=1$.
We divide the sum as
\begin{align}\label{pr-cor}
&\frac{1}{\pi(X)} \sum_{\substack{q^m \leq X\\ q :\mathrm{prime}\\ m\geq 1}} A(q^m)
\notag\\
&=\frac{1}{\pi(X)}\sum_{q\leq X}A(q)+
\frac{1}{\pi(X)}\sum_{2\leq m\leq [\log X/\log 2]}\sum_{q\leq X^{1/m}}A(q^m).
\end{align}
Using \eqref{P1} and the fact that $\Psi$ is bounded, we find that $A(q^m)$ is
bounded.    Hence the second term on the right-hand side of \eqref{pr-cor} is
\[
\ll \frac{1}{\pi(X)}\sum_m \pi(X^{1/m})
\leq \frac{1}{\pi(X)}\sum_m \pi(X^{1/2})
\ll X^{-1/2}\log X,
\]
which tends to 0 as $X\to\infty$.   
Lastly we apply the case $m=1$ of the first part of the corollary to
the first term on the right-hand side of \eqref{pr-cor} to find that it tends 
to the desired integral. 
\end{proof}

%
%
\section{The density function ${\mathcal{M}}_{\sigma}$.}

Now we start the proof of our main theorem.
In this section we first construct the density function 
${\mathcal{M}}_{\sigma, \mathcal{P}}$
for a finite set $\mathcal{P} \subset \mathbb{P}(q)$.
By $|\mathcal{P}|$ we denote the number of the elements of $\mathcal{P}$.
%
%
\begin{prop}~\label{M_P}
For any $\sigma >0$, there exists a non-negative function
${\mathcal{M}}_{\sigma, \mathcal{P}}$ defined on $\mathbb{R}$
which satisfies following two properties.
\begin{itemize}
\item The support of ${\mathcal{M}}_{\sigma, \mathcal{P}}$ is compact.
\item For any continuous function $\Psi$ on $\mathbb{R}$
, we have
\[
\int_{\mathbb{R}}
{\mathcal{M}}_{\sigma, \mathcal{P}}(u)\Psi(u) \frac{du}{\sqrt{2\pi}}
=
\int_{\mathcal{T}_{\mathcal{P}}} 
\Psi(2\Re(g_{\sigma, \mathcal{P}}( t_{\mathcal{P}})))
d^{*} t_{\mathcal{P}},
\]
where $d^* t_{\mathcal{P}}$ is the normalized 
Haar measure of $\displaystyle \mathcal{T}_{\mathcal{P}}$.
In particular, taking $\Psi\equiv 1$, we have
$\displaystyle \int_{\mathbb{R}}{\mathcal{M}}_{\sigma, \mathcal{P}}(u) \frac{du}{\sqrt{2\pi}}=1$. 
\end{itemize}
\end{prop}

\begin{proof}
We construct the function ${\mathcal{M}}_{\sigma, \mathcal{P}}$ by using the method
similar to that in 
Ihara and the first author~\cite{i-1stAuthor}.

In the case $|\mathcal{P}|=1$ namely $\mathcal{P}=\{p\}$, 
we define a one-to-one correspondence from the 
open set $(-\pi, 0)$ to its image $A(\sigma, p)\subset \mathbb{R}$ 
by 
\[
u=u(\theta)=-2\log|1-e^{i\theta}p^{-\sigma}|
\]
for $\theta\in (-\pi, 0)$. 
In fact, since
\begin{align}\label{dudtheta}
\frac{du}{d\theta}=-\frac{2p^{-\sigma}\sin(\theta)}
{|1-e^{i\theta}p^{-\sigma}|^2},
\end{align}
we see that $u$ is monotonically incrasing with respect to $\theta$, 
hence one to one.
The definition of ${\mathcal{M}}_{\sigma, \mathcal{P}}={\mathcal{M}}_{\sigma, p}$ is
%
%
\[
{\mathcal{M}}_{\sigma, p}(u)
=
\begin{cases}
\displaystyle \frac{|1-e^{i\theta}p^{-\sigma}|^2}{-\sqrt{2\pi}\sin(\theta)p^{-\sigma}} 
    & u \in A(\sigma,  p),\\
  0 & \mbox{otherwise}.
\end{cases}
\]
This function satisfies the properties of Proposition~\ref{M_P}.
In fact, using \eqref{dudtheta}
we obtain
\begin{align*}
&
\int_{\mathbb{R}}
\Psi(u){\mathcal{M}}_{\sigma, p}(u) \frac{du}{\sqrt{2\pi}}
=
\int_{A(\sigma,  p)}
\Psi(u){\mathcal{M}}_{\sigma, p}(u) \frac{du}{\sqrt{2\pi}}
\\
=&
\lim_{t_1, t_2 \to 0}
\int_{-\pi+t_1}^{-t_2}
\Psi(-2\log|1-e^{i\theta} p^{-\sigma}|)
\\
&\times
\frac{|1-e^{i\theta} p^{-\sigma}|^2}
     {(-\sqrt{2\pi}\sin\theta p^{-\sigma})}
\cdot
\frac{(-2\sin(\theta) p^{-\sigma})}
     {|1-e^{i\theta} p^{-\sigma}|^2}
\cdot
\frac{d\theta}{\sqrt{2\pi}}
\\
=&
\lim_{t_1, t_2 \to 0}
\frac{1}{\pi}
\int_{-\pi+t_1}^{-t_2}
\Psi(-2\log|1-e^{i\theta} p^{-\sigma}|)
d\theta
\\
=&
\frac{1}{\pi}
\int_0^{\pi}
\Psi(-2\log|1-e^{i\theta} p^{-\sigma}|)
d\theta
\\
=&
\frac{1}{2\pi}
\int_{-\pi}^{\pi}
\Psi(-2\log|1-e^{i\theta} p^{-\sigma}|)
d\theta
\\
=&
\int_{\mathcal{T}_p}
\Psi(-2\log|1-t_p p^{-\sigma}|)
d^*t_{p}
\\
=&
\int_{\mathcal{T}_{p}}
\Psi(2\Re(g_{\sigma, p}(t_{p})))
d^*t_{p}.
\end{align*}

%
%
In the case $|\mathcal{P}|>1$, 
we construct the function ${\mathcal{M}}_{\sigma, \mathcal{P}}$ by 
the convolution product of ${\mathcal{M}}_{\sigma, \mathcal{P}'}$ 
and ${\mathcal{M}}_{\sigma, p}$
for $\mathcal{P}=\mathcal{P}'\cup\{p\}$ inductively, that is
\[
{\mathcal{M}}_{\sigma, \mathcal{P}}(u)
=\int_{\mathbb{R}}{\mathcal{M}}_{\sigma, \mathcal{P}'}(u'){\mathcal{M}}_{\sigma, p}(u-u') \frac{du'}{\sqrt{2\pi}}.
\]
It is easy to show that 
this function satisfies the statements of Proposition~\ref{M_P}.
\end{proof}

%
%
Secondly, for the purpose of considering 
$\displaystyle \lim_{|\mathcal{P}|\to \infty} {\mathcal{M}}_{\sigma, \mathcal{P}}$, 
we define the Fourier 
transform of ${\mathcal{M}}_{\sigma, \mathcal{P}}$. 

When $\mathcal{P}=\{p\}$, we define $\widetilde{\mathcal{M}}_{\sigma, p}$ by
\[
\widetilde{\mathcal{M}}_{\sigma, p}(x)
=\int_{\mathbb{R}}{\mathcal{M}}_{\sigma, p}(u)\psi_x(u) \frac{du}{\sqrt{2\pi}},
\]
where $\psi_x(u)=e^{ixu}$.
(The Fourier transform is sometimes defined by using $e^{-ixu}$ instead of
$e^{ixu}$, but here we follow the notation in \cite{ihara} and
\cite{i-1stAuthor}.)
As Ihara and the first author discussed in p.644 of \cite{i-1stAuthor}, 
we can show  
\[
\widetilde{\mathcal{M}}_{\sigma, p}(x)=O\left((1+|x|)^{-1/2}\right)
\]
by using the Jessen-Wintner Theorem \cite{jw}.
We define $\widetilde{\mathcal{M}}_{\sigma, \mathcal{P}}(x)$ by
\[
\widetilde{\mathcal{M}}_{\sigma, \mathcal{P}}(x)
=\prod_{p \in \mathcal{P}}  \widetilde{\mathcal{M}}_{\sigma, p}(x).
\]
Then we have
\begin{equation}\label{JW_P}
\widetilde{\mathcal{M}}_{\sigma, \mathcal{P}}(x) \ll \left((1+|x|)^{-|\mathcal{P}|/2}\right)
\end{equation}
from the above estimate of $\widetilde{\mathcal{M}}_{\sigma, p}$.
On the other hand, we have the trivial bound
\begin{equation}\label{trivial_P}
\left|\widetilde{\mathcal{M}}_{\sigma, \mathcal{P}}(x)\right|\leq 1.
\end{equation}

%
%
We can show the following properties (a), (b) and (c).
The proofs of them are also similar to \cite{i-1stAuthor}, pp.645-646.
\begin{itemize}
\item[(a).] $\widetilde{\mathcal{M}}_{\sigma, \mathcal{P}}(x)\in L^t$ $(t\in [1,\infty])$.
\item[(b).] For any subsets $\mathcal{P}'$ and $\mathcal{P}$  
of $\mathbb{P}(q)$ with $\mathcal{P}'\subset \mathcal{P}$,
from \eqref{trivial_P} we can see 
\[
\left|\widetilde{\mathcal{M}}_{\sigma, \mathcal{P}}(x)\right| \leq 
\left|\widetilde{\mathcal{M}}_{\sigma, \mathcal{P}'}(x)\right|.
\]

\item[(c).] Let $y\in\mathbb{N}$, and put
$\mathcal{P}_y=\{ p \in \mathbb{P}(q) \mid  p \leq y \} \subset \mathbb{P}(q)$.
We can show the existence of
$\displaystyle \lim_{y\to\infty} \widetilde{\mathcal{M}}_{\sigma, \mathcal{P}_y}(x)$
for $\sigma >1/2$.
We denote it by $\widetilde{\mathcal{M}}_{\sigma}(x)$.
For any $a >0$, this convergence is uniform on $|x|\leq a$.
\end{itemize}

These properties yield the next proposition which is the analogue of Proposition 3.4 in
\cite{i-1stAuthor}.

%
%
\begin{prop}\label{tildeM}
For $\varepsilon>0$ and $\sigma \geq 1/2+\varepsilon$,
there exists
\[
\widetilde{\mathcal{M}}_{\sigma}(x)
=\lim_{y\to\infty}\widetilde{\mathcal{M}}_{\sigma, \mathcal{P}_y}(x),
\]
whose convergence is uniform in $x\in \mathbb{R}$.
For each $\sigma >1/2$, 
the above convergence is $L^t$-convergence and 
the function $\widetilde{\mathcal{M}}_{\sigma}(x)$ belongs to 
$ L^t$ $(1\leq t \leq \infty)$.
\end{prop}

By using \eqref{JW_P} and \eqref{trivial_P}, we have
\begin{equation}\label{tildeJW}
\widetilde{\mathcal{M}}_{\sigma}(x)=O\left( (1+|x|)^{-n/2}\right)
\end{equation}
for any $n\in\mathbb{N}$. 
We also have
\begin{equation}\label{tildeTrivial}
|\widetilde{\mathcal{M}}_{\sigma}(x)|\leq 1.
\end{equation}

Finally, we define the function ${\mathcal{M}}_{\sigma}(u)$.
For any finite set $\mathcal{P} \subset \mathbb{P}(q)$, 
we have
\[
\int_{\mathbb{R}}
\widetilde{\mathcal{M}}_{\sigma, \mathcal{P}}(x)\psi_{-u}(x) 
\frac{dx}{\sqrt{2\pi}}
=
{\mathcal{M}}_{\sigma, \mathcal{P}}(u).
\]
This is the Fourier inverse transform.
We define
\[
{\mathcal{M}}_{\sigma}(u)
=\int_{\mathbb{R}}\widetilde{\mathcal{M}}_{\sigma}(x)\psi_{-u}(x)\frac{dx}{\sqrt{2\pi}},
\]
where we can see that the right-hand side of this equation is convergent
by using \eqref{tildeJW}.

\begin{prop}\label{M}
For $\sigma>1/2$, the function ${\mathcal{M}}_{\sigma}$ satisfies following
five properties.
\begin{itemize}
\item $\displaystyle \lim_{y\to\infty} {\mathcal{M}}_{\sigma, \mathcal{P}_y}(u)
={\mathcal{M}}_{\sigma}(u)$
and this convergence is uniform in $u$.
\item The function ${\mathcal{M}}_{\sigma}(u)$ is continuous in $u$ and non-negative.
\item $\displaystyle \lim_{u \to\infty} {\mathcal{M}}_{\sigma}(u)=0$.
\item The functions ${\mathcal{M}}_{\sigma}(u)$ and $\widetilde{\mathcal{M}}_{\sigma}(x)$ 
are Fourier duals of each other.
\item $\displaystyle \int_{\mathbb{R}} {\mathcal{M}}_{\sigma}(u)\frac{du}{\sqrt{2\pi}}=1$.
\end{itemize}
\end{prop}

This is the analogue of Proposition 3.5 in \cite{i-1stAuthor} and
the proof is similar.

%
%
\section{The key lemma.}


For a fixed $\sigma>1/2$, $\tau\in\mathbb{R}$ and a finite set 
$\mathcal{P}\subset \mathbb{P}(q)$,
we put
\[
\Phi_{\sigma, \tau, \mathcal{P}}(t_{\mathcal{P}}, t'_{\mathcal{P}})
=\sum_{ p\in \mathcal{P}}
(g_{\sigma,  p}(t_p p^{-i\tau})+g_{\sigma,  p}({t'}_p p^{-i\tau})),
\]
where $t_{\mathcal{P}}=(t_p)_{p\in\mathcal{P}},
t'_{\mathcal{P}}=({t'}_p)_{p\in\mathcal{P}}\in\mathcal{T}_{\mathcal{P}}$.
From \eqref{g_sigma} we see that
\[
\psi_x\circ \Phi_{\sigma, \tau, \mathcal{P}}(\alpha_f^{\mu}(\mathcal{P}), \beta_f^{\mu}(\mathcal{P}))
=
\psi_x(\log L_{\mathcal{P}}(\mathrm{Sym}_f^{\mu}, \mathrm{Sym}_f^{\nu}, \sigma+i\tau)),
\]
where $\psi_x(u)=\exp(ixu)$
.
Therefore, to prove our Theorem~\ref{main}, it is important to consider two averages 
\begin{align*}
\mathrm{avg}_{\mathrm{prime}}
(\psi_x \circ \Phi_{\sigma, \tau, \mathcal{P}})
=&
\lim_{\substack{q\to\infty\\ q:\mathrm{prime}\\ m:\mathrm{fix}}}
\sum_{f\in S_k(q^m)}^{\qquad\prime}
\psi_x\circ \Phi_{\sigma, \tau, \mathcal{P}}(\alpha_f^{\mu}(\mathcal{P}), \beta_f^{\mu}(\mathcal{P}))
\end{align*}
and
\begin{align*}
\mathrm{avg}_{\mathrm{power}}(\psi_x \circ \Phi_{\sigma, \tau, \mathcal{P}})
=&
\lim_{\substack{m\to\infty\\ q:\mathrm{fixed}\; \mathrm{prime}\\ (\mathrm{If}\; 1\geq \sigma > 1/2,\; q\geq Q(\mu))}}
\sum_{f\in S_k(q^m)}^{\qquad\prime}
\psi_x\circ \Phi_{\sigma, \tau, \mathcal{P}}(\alpha_f^{\mu}(\mathcal{P}), \beta_f^{\mu}(\mathcal{P})).
\end{align*}

Our first aim in this section is to show the following

%
%
\begin{lem}\label{key}
Let $\mu > \nu \geq 1$ be integers with $\mu-\nu=2$ and
$\mathcal{P}$ be a finite subset of $\mathbb{P}(q)$. 
In the case $2\leq k < 12$ or $k=14$, we have
\begin{align*}
\mathrm{avg}_{\mathrm{prime}} (\psi_x \circ \Phi_{\sigma, \tau, \mathcal{P}})
=&
\mathrm{avg}_{\mathrm{power}} (\psi_x \circ \Phi_{\sigma, \tau, \mathcal{P}})
\\
=&
\int_{\mathcal{T}_{\mathcal{P}}}
\psi_x 
\left(
\Phi_{\sigma, \tau, \mathcal{P}}(t_{\mathcal{P}}, t_{\mathcal{P}}^{-1})
\right)
d^* t_{\mathcal{P}}.
\end{align*}
The above convergence is uniform in $|x|\leq R$ for any $R>0$.
\end{lem}

\begin{proof}
Let $1>\varepsilon'>0$.    Considering the Taylor expansion we find that
there exist an $M_{p}=M_p(\varepsilon',R)\in\mathbb{N}$
and $d_{m_{p}}\in\mathbb{C}$ $(0\leq m_{p} \leq M_{p})$ such that
$\psi_x\circ g_{\sigma, p}$ can be approximated by
a polynomial as
\begin{equation}\label{app-uni}
\bigg|
\psi_x \circ g_{\sigma, p}(t_{p})
- 
\sum_{m_{p}=0}^{M_{p}} d_{m_{p}} t_{p}^{m_{p}}
\bigg|
<{\varepsilon'},
\end{equation}
uniformly on $\mathcal{T}$ with respect to $t_p$ and also on $|x|\leq R$ with respect
to $x$.
Replacing $t_p$ by $t_p p^{-i\tau}$, we have
\[
\bigg|
\psi_x \circ g_{\sigma, p}(t_{p} p^{-i\tau})
- 
\sum_{m_{p=0}}^{M_{p}} c_{m_{p}} t_{p}^{m_{p}}
\bigg|
<{\varepsilon'},
\]
where $c_{m_{p}}=d_{m_{p}} p^{-i\tau m_{p}}$.
Write
\[
\varPsi_{\sigma, \tau, p}(t_p;M_p)=
\sum_{m_{p=0}}^{M_{p}} c_{m_p} t_p^{m_p}
\]
and define
\[
\Psi_{\sigma, \tau, \mathcal{P}}(t_{\mathcal{P}},  t'_{\mathcal{P}}
;M_{\mathcal{P}})=
\prod_{p \in \mathcal{P}}
\varPsi_{\sigma, \tau, p}(t_p;M_p)\varPsi_{\sigma, \tau, p}
({t'}_p;M_p),
\]
where $M_{\mathcal{P}}=(M_p)_{p\in\mathcal{P}}$.

Let $\varepsilon''>0$.
Choosing $\varepsilon'$ (depending
on $|\mathcal{P}|$ and $\varepsilon''$) sufficiently small, we obtain 
\begin{equation}\label{app}
|
\psi_x \circ \Phi_{\sigma, \tau, \mathcal{P}}(t_{\mathcal{P}}, t'_{\mathcal{P}})
-
\Psi_{\sigma, \tau, \mathcal{P}}(t_{\mathcal{P}}, t'_{\mathcal{P}};M_{\mathcal{P}})| 
< \varepsilon'',
\end{equation}
again uniformly on $\mathcal{T}$ with respect to $t_p$ and also on $|x|\leq R$ 
with respect to $x$.
In fact, since
\begin{align*}
\lefteqn{\psi_x \circ \Phi_{\sigma, \tau, \mathcal{P}}(t_{\mathcal{P}}, 
t'_{\mathcal{P}})}\\
&=\prod_{p\in\mathcal{P}}\psi_x(g_{\sigma,p}(t_p p^{-i\tau}))
\psi_x(g_{\sigma,p}(t_p^{\prime}p^{-i\tau}))\\
&=\prod_{p\in\mathcal{P}}(\Psi_{\sigma,\tau,p}(t_p;M_p)
+O(\varepsilon'))(\Psi_{\sigma,\tau,p}(t_p^{\prime};M_p)
+O(\varepsilon'))\\
&=\Psi_{\sigma,\tau,\mathcal{P}}(t_{\mathcal{P}},t_{\mathcal{P}}^{\prime}
;M_{\mathcal{P}})
+({\rm remainder \;\;terms}),
\end{align*}
we obtain \eqref{app}.
%
%

The first step of the proof of the lemma is to
express the average of the value of $\psi_x\circ \Phi_{\sigma, \tau, \mathcal{P}}$ 
by using $\Psi_{\sigma, \tau, \mathcal{P}}$.   Let $\varepsilon>0$.
From \eqref{app} with $t_{\mathcal{P}}=\alpha_f^{\mu}(\mathcal{P})$, 
$t_{\mathcal{P}}^{\prime}=\beta_f^{\mu}(\mathcal{P})$, 
$\varepsilon''=\varepsilon/2$
and \eqref{P1}, we have 
\begin{align}\label{psiPhi-Psi}
&
\left|
\sum_{f\in S_k(q^m)}^{\qquad\prime}
\psi_x \circ \Phi_{\sigma, \tau, \mathcal{P}}(\alpha_f^{\mu}(\mathcal{P}), \beta_f^{\mu}(\mathcal{P}))
-
\sum_{f\in S_k(q^m)}^{\qquad\prime}
\Psi_{\sigma, \tau, \mathcal{P}}(\alpha_f^{\mu}(\mathcal{P}), \beta_f^{\mu}(\mathcal{P})
;M_{\mathcal{P}})
\right|
\nonumber\\
< &
\sum_{f\in S_k(q^m)}^{\qquad\prime}
\varepsilon''
=\frac{\varepsilon}{2}(1 +O(E(q^m))).
\end{align}
Therefore, if $q^m$ is sufficiently large, from \eqref{E2}
we see that
\begin{align}\label{step-1}
\left|
\sum_{f \in S(q^m)}^{\qquad\prime} 
\psi_x \circ \Phi_{\sigma, \tau, \mathcal{P}}(\alpha_f^{\mu}(\mathcal{P}), \beta_f^{\mu}(\mathcal{P}))
-
\sum_{f\in S(q^m)}^{\qquad\prime} \Psi_{\sigma, \tau, \mathcal{P}}(\alpha_f^{\mu}(\mathcal{P}), \beta_f^{\mu}(\mathcal{P});M_{\mathcal{P}})
\right|<\varepsilon.
\end{align}

%
%

As the second step, we calculate
$\Psi_{\sigma, \tau, \mathcal{P}}$ as follows;
\begin{align*}
&
\Psi_{\sigma, \tau, \mathcal{P}}(\alpha_f^{\mu}(\mathcal{P}), \beta_f^{\mu}(\mathcal{P})
;M_{\mathcal{P}})
\\
=&
\prod_{ p\in \mathcal{P}}
\Big(\sum_{m_{p}=0}^{M_{p}}c_{m_{p}}e^{\mu im_{p}\theta_f(p)}\Big)
\Big(\sum_{n_{p}=0}^{M_{p}}c_{n_{p}}e^{-\mu in_{p}\theta_f(p)}\Big)
\\
=&
\prod_{p\in \mathcal{P}}
\Big(
\sum_{m_{p}=0}^{M_{p}}c_{m_{p}}^2
\\
&
+
\sum_{m_{p}=0}^{M_{p}}
\sum_{\substack{n_{p}=0\\m_{p}<n_{p}}}^{M_{p}}
c_{m_{p}}e^{\mu im_{p}\theta_f(p)}
c_{n_{p}}e^{-\mu in_{p}\theta_f(p)}
\\
&+
\sum_{m_{p}=0}^{M_{p}}
\sum_{\substack{n_{p}=0\\m_{p}> n_{p}}}^{M_p}
c_{m_{p}}e^{\mu im_{p}\theta_f(p)}
c_{n_{p}}e^{-\mu in_{p}\theta_f(p)}
\Big)
\\
=&
\prod_{p \in \mathcal{P}}
\Big(
\sum_{m_{p}=0}^{M_{p}}c_{m_{p}}^2
\\
&+
\sum_{m_{p}=0}^{M_{p}}
\sum_{\substack{n_{p}=0 \\ m_{p}<n_{p}}}^{M_{p}}
c_{m_{p}}c_{n_{p}}
\big(e^{\mu i(m_{p}-n_{p})\theta_f(p)}+
e^{\mu i(n_{p}-m_{p})\theta_f(p)}
\big)
\Big).
\end{align*}
We put $n_{p}-m_{p}=r_{p}$.
Using \eqref{euler}, we see that
\begin{align*}
&
\Psi_{\sigma, \tau, \mathcal{P}}(\alpha_f^{\mu}(\mathcal{P}), \beta_f^{\mu}(\mathcal{P})
;M_{\mathcal{P}})
\\
=&
\prod_{ p \in \mathcal{P}}
\bigg(
\sum_{m_{p}=0}^{M_{p}}c_{m_{p}}^2
+
\sum_{r_{p}=1}^{M_{p}}
\sum_{n_{p}=r_{p}}^{M_p}
c_{n_{p}-r_{p}}c_{n_{p}}
\big(
e^{\mu ir_{p}\theta_f(p)}
+
e^{-\mu ir_{p}\theta_f(p)}
\big)
\bigg)
\\
=&
\prod_{p \in \mathcal{P}}
\Bigg(
\sum_{m_{p}=0}^{M_{p}}c_{m_{p}}^2
+
\sum_{r_{p}=1}^{M_{p}}
\sum_{n_{p}=r_{p}}^{M_p}
c_{n_{p}-r_{p}}c_{n_{p}}
\Big((e^{\mu ir_{p}\theta_f(p)}
\\&
+
\sum_{\ell=1}^{\mu r_{p}-1}e^{i( \mu r_{p}-2\ell)\theta_f(p)}
+
e^{- \mu ir_{p}\theta_f(p)})
-
\sum_{\ell=1}^{\mu r_{p}-1}e^{i(\mu r_{p}-2\ell)\theta_f(p)}
\Big)
\Bigg)
\\
=&
\prod_{p\in \mathcal{P}}
\Bigg(
\sum_{m_{p}=0}^{M_{p}}c_{m_{p}}^2
+
\sum_{r_{p}=1}^{M_{p}}
\sum_{n_{p}=r_{p}}^{M_p}
c_{n_{p}-r_{p}}c_{n_{p}}
(
\lambda_f( p^{\mu r_{p}})
-
\lambda_f( p^{\mu r_{p}-2})
)
\Bigg).
\end{align*}
Since $\mu = \nu+2 \geq 3$, by using \eqref{P}, we obtain
\begin{align*}
&
\sum_{f\in S(q^m)}^{\qquad\prime} \Psi_{\sigma, \tau, \mathcal{P}}
(\alpha_f^{\mu}(\mathcal{P}), \beta_f^{\mu}(\mathcal{P});M_{\mathcal{P}})
\\
=
&
\sum_{f\in S_k(q^m)}^{\qquad\prime}
\prod_{ p\in \mathcal{P}}
\bigg(
\sum_{m_{p}=0}^{M_{p}}c_{m_{p}}^2
+
\sum_{r_{p}=1}^{M_{p}}
\sum_{n_{p}=r_{p}}^{M_p}
c_{n_{p}-r_{p}}c_{n_{p}}
(
\lambda_f( p^{\mu r_{p}})
-
\lambda_f(p^{\mu r_{p}-2})
)
\bigg)
\\
=&
\prod_{p \in \mathcal{P}}
\sum_{m_{p}=0}^{M_{p}}c_{m_{p}}^2
+ O(E(q^m)),
\end{align*}
where the implied constant of the error term depends on 
$\mathcal{P}$, $\mu$ and $M_{\mathcal{P}}=M_{\mathcal{P}}(\varepsilon',R)$
(hence depends on $\varepsilon$ under the above choice of $\varepsilon'$).
But still, 
this error term can be smaller than $\varepsilon$ for sufficiently large $q^m$.
Combining this with \eqref{step-1}, we obtain
\begin{align}\label{step-2}
\left|
\sum_{f \in S(q^m)}^{\qquad\prime} 
\psi_x \circ \Phi_{\sigma, \tau, \mathcal{P}}(\alpha_f(\mathcal{P})^{\mu}, \beta_f(\mathcal{P})^{\mu})
-
\prod_{p \in \mathcal{P}}
\sum_{m_{p}=0}^{M_{p}}c_{m_{p}}^2
\right|<2\varepsilon.
\end{align}


As the final step, we calculate the integral 
in the statement of Lemma~\ref{key}.
For any $\varepsilon>0$, using \eqref{app}, we have
\begin{align}\label{step-3}
&
\int_{\mathcal{T}_{\mathcal{P}}} 
\psi_x(
\Phi_{\sigma, \tau, \mathcal{P}}( t_{\mathcal{P}},  t_{\mathcal{P}}^{-1})
)
d^* t_{\mathcal{P}} 
\nonumber\\
=&
\int_{\mathcal{T}_{\mathcal{P}}} 
(
\psi_x \left(
\Phi_{\sigma, \tau, \mathcal{P}}( t_{\mathcal{P}},  t_{\mathcal{P}}^{-1})
\right)
-
\Psi_{\sigma, \tau, \mathcal{P}}( t_{\mathcal{P}},  t_{\mathcal{P}}^{-1}
;M_{\mathcal{P}})
+
\Psi_{\sigma, \tau, \mathcal{P}}( t_{\mathcal{P}},  t_{\mathcal{P}}^{-1}
;M_{\mathcal{P}})
)
d^* t_{\mathcal{P}} 
\nonumber\\
=&
\int_{\mathcal{T}_{\mathcal{P}}} 
\Psi_{\sigma, \tau, \mathcal{P}}( t_{\mathcal{P}},  t_{\mathcal{P}}^{-1}
;M_{\mathcal{P}})
d^* t_{\mathcal{P}} 
+O(\varepsilon)
\nonumber\\
=&
\int_{\mathcal{T}_{\mathcal{P}}}
\prod_{ p\in \mathcal{P}}
\bigg(
\sum_{m_{p}=0}^{M_{p}}c_{m_{p}}t_{p}^{m_{p}}
\bigg)
\bigg(
\sum_{n_{p}=0}^{M_{p}}c_{n_{p}}{t_{p}}^{-n_{p}}
\bigg)
d^* t_{\mathcal{P}} 
+O(\varepsilon)
\nonumber\\
=&
\prod_{ p\in \mathcal{P}}
\sum_{m_{p}=0}^{M_{p}}c_{m_{p}}^2
+O(\varepsilon).
\end{align}
From \eqref{step-2} and \eqref{step-3} we find that the identity in the statement 
of Lemma~\ref{key}
holds with the error $O(\varepsilon)$,
but this error can be arbitrarily small,
so the assertion of Lemma~\ref{key} follows.
\end{proof}

\begin{rem}\label{reason}
In the above proof, the function 
$\Psi_{\sigma, \tau, \mathcal{P}}(\alpha_f^{\mu}(\mathcal{P}), \beta_f^{\mu}(\mathcal{P})
;M_{\mathcal{P}})$ 
is expressed by 
\[
e^{\mu im_p\theta_f(p)}e^{-\mu in_p\theta_f(p)},
\qquad (m_p, n_p \geq 0,\; \mu \geq 3).
\]
When $m_p \neq n_p$ these are written by
$\lambda_f(p^{\mu r_p})$
and
$\lambda_f(p^{\mu r_p-2})$
$(r_p \geq 1)$
and, as shown above, they are included in the error terms by \eqref{P}.
If we try to study averages of $\log L_{\mathbb{P}(q)}(\mathrm{Sym}_f^{\gamma}, s)$ 
itself (without considering the difference)
by the same method as in this paper, 
we have to handle the terms of the form
\[
e^{im_p\theta_f(p)}e^{-in_p\theta_f(p)}\qquad (m_p-n_p=\pm 2). 
\]
However, these terms produce other ``main'' terms by \eqref{P1}, since 
$\alpha_f^2(p)+\beta_f^2(p)=\lambda_f(p^2)-1$.
This invalidates the above argument, so
our method, as it is, cannot be applied to $\log L(\mathrm{Sym}_f^{\gamma}, s)$.
\end{rem}

When $\tau=0$, 
Proposition~\ref{M_P} and Lemma~\ref{key} imply
\begin{align}\label{key-tau=0}
&
\mathrm{avg}_{\mathrm{prime}} (\psi_x \circ \Phi_{\sigma, 0, \mathcal{P}})
=
\mathrm{avg}_{\mathrm{power}} (\psi_x \circ \Phi_{\sigma, 0, \mathcal{P}})
\nonumber\\
=&
\int_{\mathcal{T}_{\mathcal{P}}}
\psi_x (
\Phi_{\sigma, 0, \mathcal{P}}(t_{\mathcal{P}}, t_{\mathcal{P}}^{-1})
)
d^* t_{\mathcal{P}}
\nonumber\\
=&
\int_{\mathcal{T}_{\mathcal{P}}}
\psi_x \big(
\sum_{p\in\mathcal{P}}
(
g_{\sigma, p}(t_p)+g_{\sigma, p}(t_p^{-1})
)
\big)
d^* t_{\mathcal{P}}
\nonumber\\
=&
\int_{\mathcal{T}_{\mathcal{P}}}
\psi_x (
g_{\sigma, \mathcal{P}}(t_{\mathcal{P}})
+\overline{g_{\sigma,\mathcal{P}}(t_{\mathcal{P}})}
)
d^* t_{\mathcal{P}}
\nonumber\\
=&
\int_{\mathcal{T}_{\mathcal{P}}}
\psi_x (2\Re(g_{\sigma, \mathcal{P}}(t_{\mathcal{P}}))
d^* t_{\mathcal{P}}
=
\int_{\mathbb{R}}
{\mathcal{M}}_{\sigma, \mathcal{P}}(u)\psi_x (u)\frac{du}{\sqrt{2\pi}},
\end{align}
uniformly in $|x|\leq R$.
This fact deduces the case $\sigma >1$ of the following key lemma.
%
%
\begin{lem}\label{keylemma}
Let $\mu > \nu \geq 1$ be integers with $\mu-\nu=2$.
Suppose Assumption~\ref{ass} and \ref{grh}.
In the case  $2\leq k <12$ or $k=14$, for $\sigma>1/2$ and
$\psi_x(u)=\exp (ixu)$, 
we have
\begin{align*}
&
\mathrm{Avg}_{\mathrm{prime}} \psi_x (\log L_{\mathbb{P}(q)}(\mathrm{Sym}_f^{\mu}, \sigma)-\log L_{\mathbb{P}(q)}(\mathrm{Sym}_f^{\nu}, \sigma))
\\
=&
\mathrm{Avg}_{\mathrm{power}} \psi_x (\log L_{\mathbb{P}(q)}(\mathrm{Sym}_f^{\mu}, \sigma)-\log L_{\mathbb{P}(q)}(\mathrm{Sym}_f^{\nu}, \sigma))
\\
=&
\int_{\mathbb{R}}
{\mathcal{M}}_{\sigma}(u)
\psi_x(u)\frac{du}{\sqrt{2\pi}}
\end{align*}
The above convergence is uniform in $|x|\leq R$ for any $R>0$.
\end{lem}

We note that this lemma is actually a special case $\Psi=\psi_x$ in our main
Theorem~\ref{main}.    To show this lemma is the main body of the proof of Theorem~\ref{main}.

\begin{proof}[Proof in the case $\sigma >1$]
Since $\sigma >1$, we find a sufficiently large finite subset 
$\mathcal{P} \subset \mathbb{P}(q)$
for which it holds that
$$
|L_{\mathbb{P}(q)}(\mathrm{Sym}_f^{\mu}, \mathrm{Sym}_f^{\nu},s)
-
L_{\mathcal{P}}(\mathrm{Sym}_f^{\mu}, \mathrm{Sym}_f^{\nu},s)
|<\varepsilon
$$
and
$
|\widetilde{\mathcal{M}}_{\sigma, \mathcal{P}}(x)-\widetilde{\mathcal{M}}_{\sigma}(x)|<\varepsilon
$
for any $x\in\mathbb{R}$ and any $\varepsilon>0$.
The last inequality is provided by Proposition~\ref{tildeM}.
We can choose the above $\mathcal{P}$ which does not depend on $q^m$.
Using this $\mathcal{P}$, we have
\begin{align*}
&
\bigg|\sum_{f\in S_k(q^m)}^{\qquad\prime}
\psi_x (\log L_{\mathbb{P}(q)}(\mathrm{Sym}_f^{\mu}, \mathrm{Sym}_f^{\nu}, \sigma))
-
\int_{\mathbb{R}}{\mathcal{M}}_{\sigma}(u)\psi_x(u)\frac{du}{\sqrt{2\pi}}
\bigg|
\\
\leq &
\bigg|\sum_{f\in S_k(q^m)}^{\qquad\prime}
\bigg(\psi_x (\log L_{\mathbb{P}(q)}(\mathrm{Sym}_f^{\mu}, \mathrm{Sym}_f^{\nu}, \sigma))
-
\psi_x (\log L_{\mathcal{P}}(\mathrm{Sym}_f^{\mu}, \mathrm{Sym}_f^{\nu}, \sigma))
\bigg)\bigg|
\\
&+
\bigg|\sum_{f\in S_k(q^m)}^{\qquad\prime}
\psi_x (\log L_{\mathcal{P}}(\mathrm{Sym}_f^{\mu}, \mathrm{Sym}_f^{\nu}, \sigma))
-
\int_{\mathbb{R}}{\mathcal{M}}_{\sigma, \mathcal{P}}(u)\psi_x(u)\frac{du}{\sqrt{2\pi}}
\bigg|
\\
&+
\bigg|
\int_{\mathbb{R}}{\mathcal{M}}_{\sigma, \mathcal{P}}(u)\psi_x(u)\frac{du}{\sqrt{2\pi}}
-
\int_{\mathbb{R}}{\mathcal{M}}_{\sigma}(u)\psi_x(u)\frac{du}{\sqrt{2\pi}}
\bigg|
\\
&
=S_1+S_2+S_3,
\end{align*}
say.   We remind the relation 
\begin{equation}\label{ihara}
|\psi_x(u)-\psi_x(u')|
\ll |x|\cdot |u-u'|
\end{equation}
for $u\in \mathbb{R}$ (see Ihara~\cite{ihara} (6.5.19) or Ihara-Matsumoto~\cite{i-1stAuthor}).
We see that
\begin{align*}
S_1 \ll \; & |x| 
\sum_{f\in S_k(q^m)}^{\qquad\prime}
\bigg(
|\log L_{\mathbb{P}(q)}(\mathrm{Sym}_f^{\mu}, \mathrm{Sym}_f^{\nu}, \sigma)
-
\log L_{\mathcal{P}}(\mathrm{Sym}_f^{\mu}, \mathrm{Sym}_f^{\nu}, \sigma)|
\bigg)
\end{align*}
and
\[
S_3
=
\left|
\widetilde{\mathcal{M}}_{\sigma,\mathcal{P}}(x)
-
\widetilde{\mathcal{M}}_{\sigma}(x)
\right|.
\]
Therefore $|S_1|$ and $|S_3|$ are $O(\varepsilon)$
for large $|\mathcal{P}|$, with the implied constant depending on $R$.
As for the estimate on $|S_2|$, we use \eqref{key-tau=0}, whose convergence is
uniform on $|x|\leq R$.   This completes the proof.

%
\end{proof}

In the next two sections we will give the proof of Lemma~\ref{keylemma}
when $1 \geq \sigma > 1/2$. 

%
%
\section{The approximation of $L_{\mathbb{P}(q)}$ under GRH.}

In this section, we suppose Assumptions~\ref{ass} and \ref{grh}.
This section is the first step of the proof of Lemma~\ref{keylemma}
for $1 \geq \sigma > 1/2$.  

In this section, 
we study the approximation of
$L_{\mathbb{P}(q)}(\mathrm{Sym}_f^{\gamma}, s)$ by
$L_{\mathcal{P}}(\mathrm{Sym}_f^{\gamma}, s)$ 
with suitable $\mathcal{P}$ which depends on the
level of the primitive form $f$
(see Lemma~\ref{appSymL} below). 
Recall that the level of $f$ is $q^m$ and $q$ is a prime number.

Let the sets $\mathcal{P}_{\log q^m}$ and $\mathcal{P}_{\log q^m}^+$ 
be defined by
\[
\mathcal{P}_{\log q^m}= \{ p \in \mathbb{P}(q) \mid  p \leq \log q^m\}
\]
and
\[
\mathcal{P}_{\log q^m}^+= \mathcal{P}_{\log q^m} \cup \{q\}.
\]
Then
\[
\log L_{\mathbb{P}(q)}(\mathrm{Sym}_f^{\gamma}, s)
=\log L(\mathrm{Sym}_f^{\gamma}, s)+\log (1-\lambda_f(q^{\gamma})q^{-s}),
\]
\[
\log L_{\mathcal{P}_{\log q^m}}(\mathrm{Sym}_f^{\gamma}, s)
=\log L_{\mathcal{P}_{\log q^m}^+}(\mathrm{Sym}_f^{\gamma}, s) 
+\log (1-\lambda_f(q^{\gamma})q^{-s})
\]
and
\begin{align*}
&\log L_{\mathbb{P}(q)}(\mathrm{Sym}_f^{\gamma}, s)
- \log L_{\mathcal{P}_{\log q^m}}(\mathrm{Sym}_f^{\gamma}, s)
\\
=&
\log L(\mathrm{Sym}_f^{\gamma}, s)
- \log L_{\mathcal{P}_{\log q^m}^+}(\mathrm{Sym}_f^{\gamma}, s).
\end{align*}
Define
\[
F(\mathrm{Sym}_f^{\gamma}, s)
=
\frac{L_{\mathbb{P}(q)}(\mathrm{Sym}_f^{\gamma}, s)}{L_{\mathcal{P}_{\log q^m}}(\mathrm{Sym}_f^{\gamma}, s)}.
\]
Since
\begin{align*}
&
\log L(\mathrm{Sym}_f^{\gamma}, s)
\\
=&
-\log(1-\lambda_f(q^{\gamma})q^{-s})
-\sum_{ p\neq q}\sum_{h=0}^{\gamma} \log(1-\alpha_f^{\gamma -h}(p) \beta^h(p) p^{-s})
\\
=&
\sum_{\ell=1}^{\infty}
\frac{\lambda_f^{\ell}(q^{\gamma})}{\ell q^{\ell s}}
+\sum_{p \neq q}\sum_{h=0}^{\gamma} \sum_{\ell=1}^{\infty}
\frac{\alpha_f^{(\gamma-h)\ell}(p) \beta^{h\ell}(p)}{\ell p^{\ell s}}
\end{align*}
for $\sigma>1$, 
we can write
\begin{align}\label{tsuika0}
\log F(\mathrm{Sym}_f^n, s)
=
\sum_{ p\in \mathbb{P}(q)\setminus \mathcal{P}_{\log q^m}}
\sum_{\ell=1}^{\infty}\frac{c_{f,\gamma, p}(\ell)}{\ell  p^{\ell s}}
\end{align}
and
\begin{align}\label{tsuika}
\frac{F'}{F}(\mathrm{Sym}_f^{\gamma}, s)
=
-
\sum_{ p \in \mathbb{P}(q)\setminus \mathcal{P}_{\log q^m}}
\sum_{\ell=1}^{\infty}\frac{c_{f, \gamma, p}(\ell)\log  p}{ p^{\ell s}}
\end{align}
for $\sigma >1$, where the coefficients $c_{f,\gamma, p}$  
are defined by 
\begin{align}\label{tsuika2}
c_{f,\gamma, p}(\ell)=
\sum_{h=0}^{\gamma} \alpha_f^{(\gamma -h)\ell}( p)\beta^{h\ell}( p),
\end{align}
for $p \neq q$.
By Assumptions~\ref{ass} and \ref{grh}, the functions
$\log L(\mathrm{Sym}_f^{\gamma}, s)$ are holomorphic for  $\sigma >1/2$.
By the argument of the proof of Lemma~3 in Duke~\cite{d},  
we obtain
\begin{equation}\label{duke}
\left|
\frac{L'(\mathrm{Sym}_f^{\gamma}, s)}{L(\mathrm{Sym}_f^{\gamma}, s)}
\right|
\ll_{\epsilon, \gamma}
\log q^m + \log(1+|t|)
\end{equation}
for $1/2+\epsilon \leq \sigma \leq 2$ 
$(0<\epsilon \leq 1)$
from \eqref{ass-estimate} and \eqref{ass-estimate'}. 

Now we restrict ourselves to the case $\gamma=\mu$, $\nu$.    When $q \geq Q(\mu)$,
we have
\begin{align}\label{tsuika3}
&
\log L_{\mathcal{P}_{\log q^m}^+}(\mathrm{Sym}_f^{\gamma}, s)
+\log (1-\lambda_f(q^{\gamma})q^{-s})
\nonumber\\
=&
-\sum_{\substack{ p\leq \log q^m \\  p\neq q}}\sum_{h=0}^{\gamma}
\log(1-\alpha_f^{\gamma-h}( p)\beta_f^h( p) p^{-s})
\end{align}
for $\sigma>1/2$ and $\gamma=\mu$, $\nu$.
Here, on the second term of the left-hand side of the above equation,
since we know  $\lambda_f(q)<d(q)=2$ 
and 
$|\lambda_f(q^{\gamma})q^{-s}|<2^{\gamma} q^{-\sigma}< 2^{\mu}/\sqrt{q}$
from \eqref{euler}, 
we have $|\lambda_f(q^{\gamma})q^{-s}|<2^{\mu}/\sqrt{Q(\mu)}<1$
and the above logarithm is well-defined.
Differentiating the both sides of \eqref{tsuika3}, we obtain
\begin{align*}
&
\frac{L'_{\mathcal{P}_{\log q^m}^+}(\mathrm{Sym}_f^{\gamma}, s)}{L_{\mathcal{P}_{\log q^m}^+}(\mathrm{Sym}_f^{\gamma}, s)}
\\
=&
-
\frac{\lambda_f(q^{\gamma}) q^{-s} \log q}{1-\lambda_f(q^{\gamma})q^{-s}}
-\sum_{\substack{ p\leq \log q^m\\  p\neq q}} \sum_{h=0}^{\gamma}
\frac{\alpha_f^{\gamma-h}(p)\beta_f^h( p) p^{-s}\log p}{1-\alpha_f^{\gamma-h}(p)\beta_f^h( p) p^{-s}}.
\end{align*}
The denominator $(1-\lambda_f(q^{\gamma})q^{-s})$ 
has a lower bound $1-2^{\mu}/\sqrt{Q(\mu)} >0$ when $q \geq Q(\mu)$.
Therefore the first term of the right-hand side of the above equation
can be estimated by $q^{-1/2}\log q$.
Hence, when $q\geq Q(\mu)$ and $\sigma>1/2$, we have
\begin{align}\label{L'/L}
\left|\frac{L'_{\mathcal{P}_{\log q^m}^+}(\mathrm{Sym}_f^{\gamma}, s)}{L_{\mathcal{P}_{\log q^m}^+}(\mathrm{Sym}_f^{\gamma}, s)}\right|
\ll&
q^{-1/2}\log q
+
\sum_{ p\leq \log q^m}
\frac{ p^{-\sigma}\log p}{1- p^{-\sigma}}
\nonumber\\
\ll&
q^{-1/2}\log q
+
\sum_{ p\leq \log q^m}
\frac{\log p}{ p^{1/2}}
\nonumber\\
\ll&
q^{-1/2}\log q +
(\log q^m)^{1/2} 
\nonumber\\
\ll&
(\log q^m)^{1/2}
\end{align}
by partial summation and the prime number theorem.
From \eqref{duke} and \eqref{L'/L}, we obtain
\begin{equation}\label{F'/F}
\frac{F'}{F} 
(\mathrm{Sym}_f^{\gamma}, s)
\ll \log q^m+\log(1+|t|)\qquad 
(\frac{1}{2}+\varepsilon\leq\sigma\leq 2),
\end{equation}
under assumptions.

Now we assume $1/2<\sigma\leq 1$, and put $\sigma=1/2+\delta$, $0<\delta\leq 1/2$.
The following argument is similar to the proof of Proposition 5 
in Duke~\cite{d}. By Mellin's formula, we know
\[
e^{-y}=\frac{1}{2\pi i}\int_{2-i\infty}^{2+i\infty} y^{-z}\Gamma (z)dz.
\]
Therefore by \eqref{tsuika}, for $y>0$, we have
\begin{align*}
&
-
\sum_{ p\in \mathbb{P}(q)\setminus \mathcal{P}_{\log q^m}}\log p 
\sum_{\ell=1}^{\infty}\frac{c_{f,\gamma, p}(\ell)}{ p^{\ell u}} e^{- p^{\ell}/x}
\\
=&
\frac{1}{2\pi i}\int_{2-\infty}^{2+i\infty}
\frac{F'}{F}(\mathrm{Sym}_f^{\gamma}, u+z) x^z \Gamma (z) dz,
\end{align*}
where $u>0$ and $x>1$.
Now assume $(1+\delta)/2 < u\leq 3/2$.
By shifting the path of integration to $\Re z = (1+\delta)/2-u$
and using \eqref{F'/F}, 
we have
\begin{align*}
&
-\sum_{ p\in \mathbb{P}(q)\setminus \mathcal{P}_{\log q^m}} \log p 
\sum_{\ell=1}^{\infty}\frac{c_{f,\gamma, p}(\ell)}{ p^{\ell u}} e^{- p^{\ell}/x}
\\
=&
\frac{F'}{F}(\mathrm{Sym}_f^{\gamma}, u)
+O(x^{(1+\delta)/2-u} \log q^m).
\end{align*}
Integrating the above equation with respect to $u$ from $\sigma=1/2+\delta$ to $3/2$
we obtain
\begin{align*}
&
 \log F(\mathrm{Sym}_f^{\gamma}, 3/2) - \log F(\mathrm{Sym}_f^{\gamma}, \sigma) 
=
 \int_{\sigma}^{3/2}\frac{F'}{F}(\mathrm{Sym}_f^{\gamma}, u) du
\\
=&
-
\int_{\sigma}^{3/2}
\sum_{ p\in \mathbb{P}(q)\setminus \mathcal{P}_{\log q^m}} \log p 
\sum_{\ell=1}^{\infty} \frac{c_{f,\gamma, p}(\ell)}{ p^{\ell u}} e^{- p^{\ell}/x}
du
\\
&+
O\left(\int_{\sigma}^{3/2} x^{(1+\delta)/2-u} du\cdot \log q^m \right).
\end{align*}
Separating the term corresponding to $\ell=1$ on the 
right-hand side, wee see that
\begin{align}\label{tsuika4}
&
\log F(\mathrm{Sym}_f^{\gamma}, \sigma)  
-
\log F(\mathrm{Sym}_f^{\gamma}, 3/2) 
\nonumber\\
=& 
-
\sum_{ p\in \mathbb{P}(q)\setminus \mathcal{P}_{\log q^m}}
\frac{c_{f,\gamma, p}(1)}{ p^{3/2}} e^{- p/x}
+
\sum_{ p\in \mathbb{P}(q)\setminus \mathcal{P}_{\log q^m}}
\frac{c_{f,\gamma, p}(1)}{ p^{\sigma}} e^{- p/x}
\nonumber\\
&+
O_{\delta}\left(
\sum_{ p\in \mathbb{P}(q)\setminus \mathcal{P}_{\log q^m}} 
\sum_{\ell=2}^{\infty} \frac{1}{\ell p^{\ell \sigma}} e^{- p^{\ell}/x}
+
\frac{x^{1/2 -\sigma+\delta/2} \log q^m}{\log x}
\right),
\end{align}
because from \eqref{tsuika2} we see that $c_{f,\gamma, p}(\ell)=O(1)$.

On the right-hand side of \eqref{tsuika4}, we have
\begin{align*}
&
\sum_{ p\geq \log q^m} 
\sum_{\ell=2}^{\infty} \frac{1}{\ell p^{\ell \sigma}} e^{- p^{\ell}/x}
\\
\ll&
\sum_{\ell=2}^{\infty} \frac{1}{\ell}
\sum_{ p\geq \log q^m} 
\frac{1}{ p^{\ell/2+\ell\delta}} 
<
\sum_{ p\geq \log q^m} 
\frac{1}{ p^{1+2\delta}} 
+
\sum_{\ell=3}^{\infty} \frac{1}{\ell}
\sum_{ p\geq \log q^m} 
\frac{1}{ p^{\ell/2+\ell\delta}} 
\\
\ll &
\frac{1}{(\log q^m)^{\delta}}
+
\sum_{\ell=3}^{\infty} \frac{1}{\ell^2}
\sum_{ p\geq \log q^m} 
\frac{1}{ p^{\ell/3+\ell\delta}} 
\ll
\frac{1}{(\log q^m)^{\delta}}
\end{align*}
and
\[
\sum_{ p\in \mathbb{P}(q)\setminus \mathcal{P}_{\log q^m}}
\frac{c_{f,\gamma, p}(1)}{ p^{3/2}} e^{- p/x}
\ll 
\sum_{ p > \log q^m}
\frac{1}{ p^{3/2}} 
\ll 
\sum_{ n > \log q^m}
\frac{1}{ n^{3/2}} 
\ll
(\log q^m)^{-1/2}.
\]
Moreover, using \eqref{tsuika0} we have
\begin{align*}
&
|\log F(\mathrm{Sym}_f^{\gamma}, 3/2)|
\ll 
\sum_{ p\in \mathbb{P}(q)\setminus\mathcal{P}_{\log q^m}}\sum_{\ell=1}^{\infty}
\frac{1}{\ell p^{3\ell/2}}
\\
\ll &
\sum_{\ell=1}^{\infty}\frac{1}{\ell^2}
\sum_{ p\geq \log q^m}
\frac{1}{ p^{3\ell/2-\ell/4}}
<
\sum_{\log q^m < n}
\frac{1}{n^{5/4}}
\ll
(\log q^m)^{-1/4}.
\end{align*}
Letting $x= q^{m/4(k-1)\gamma}$ we obtain
\begin{align*}
&
\log F(\mathrm{Sym}_f^{\gamma}, \sigma)
-
\sum_{ p\in \mathbb{P}(q)\setminus \mathcal{P}_{\log q^m}}
\frac{c_{f,\gamma,p}(1)}{ p^{\sigma}} e^{- p/q^{m/4(k-1)\gamma}}
\\
=&
O_{\delta, k, \gamma}\left(
(\log q^m)^{-\delta}
+(\log q^m)^{-1/4}+(q^{m/4(k-1)\gamma})^{-\delta/2}
\right).
\end{align*}
Since it is easy to see that $c_{f,\gamma,p}(1)=\lambda_f( p^{\gamma})$ from
\eqref{euler} and \eqref{tsuika2}, we now
obtain the following lemma.
%
%
\begin{lem}\label{appSymL}
Suppose Assumptions~\ref{ass} and \ref{grh}.
Let $Q(\mu)$ be the smallest prime number
satisfying $2^{\mu}/\sqrt{Q(\mu)}<1$ 
and $f$ be a primitive form in $S_k(q^m)$, where $q > Q(\mu)$ is a prime.
For fixed $\gamma$ and
$\sigma=1/2 + \delta$ $(0< \delta \leq 1/2)$, 
we have
\begin{align}\label{F}
&
\log L_{\mathbb{P}(q)}(\mathrm{Sym}_f^{\gamma}, \sigma)
- 
\log L_{\mathcal{P}_{\log q^m}}(\mathrm{Sym}_f^{\gamma}, \sigma)
-
\mathcal{S}_{\gamma}
\nonumber\\
\ll &
(\log q^m)^{-\delta}
+(\log q^m)^{-1/4}+(q^{m/4(k-1)\gamma})^{-\delta/2},
\end{align}
where
\[
\mathcal{S}_{\gamma}=
\sum_{ p\in \mathbb{P}(q)\setminus \mathcal{P}_{\log q^m}}
\frac{\lambda_f( p^{\gamma})}{ p^{\sigma}} e^{- p/q^{m/(k-1)\gamma}}.
\]
\end{lem}

%
%
\section{Proof of Lemma~\ref{keylemma} for $1 \geq \sigma > 1/2$.}

We already proved Lemma~\ref{keylemma} for $\sigma > 1$ in Section 4.
In this section, we prove Lemma~\ref{keylemma} for $1 \geq \sigma > 1/2$
by using \eqref{F} proved in the previous section,
under Assumptions~\ref{ass} and \ref{grh}.
We remind the relation 
\[
\int_{\mathbb{R}}
\mathcal{M}_{\sigma}(u)
\psi_x(u)
\frac{du}{\sqrt{2\pi}}
=
\widetilde{\mathcal{M}}_{\sigma}(x).
\]

Our aim is to prove that 
\begin{equation}\label{prime}
\bigg|
\sum_{f\in S_k(q^m)}^{\qquad\prime}
\psi_{\xi}(\log L_{\mathbb{P}(q)}(\mathrm{Sym}_f^{\mu}, \sigma)-\log L_{\mathbb{P}(q)}(\mathrm{Sym}_f^{\nu}, \sigma)) 
- 
\widetilde{\mathcal{M}}_{\sigma}(x)
\bigg|
\end{equation}
tends to $0$ as $q^m\to \infty$ for fixed
$q \geq Q(\mu)$,
when $1\geq \sigma > 1/2$.

%
%

First, using \eqref{ihara}, we can see the following inequality:
\begin{align}\label{basic}
&
\bigg|
\sum_{f\in S_k(q^m)}^{\qquad\prime}
\psi_x(\log L_{\mathbb{P}(q)}(\mathrm{Sym}_f^{\mu}, \mathrm{Sym}_f^{\nu}, \sigma)) 
-
\widetilde{\mathcal{M}}_{\sigma}(x)
\bigg|
\nonumber\\
\leq&
\bigg|
\sum_{f\in S_k(q^m)}^{\qquad\prime}
\big(
\psi_x(\log L_{\mathbb{P}(q)}(\mathrm{Sym}_f^{\mu}, \mathrm{Sym}_f^{\nu}, \sigma)) 
\nonumber\\&
\hspace{12mm}-
\psi_x(\log L_{\mathcal{P}_{\log q^m}}(\mathrm{Sym}_f^{\mu}, \mathrm{Sym}_f^{\nu}, \sigma))
\big)
\bigg|
\nonumber\\
&+
\bigg|
\sum_{f\in S_k(q^m)}^{\qquad\prime}
\psi_x(\log L_{\mathcal{P}_{\log q^m}}(\mathrm{Sym}_f^{\mu}, \mathrm{Sym}_f^{\nu}, \sigma))
-
\widetilde{\mathcal{M}}_{\sigma, \mathcal{P}_{\log q^m}}(x)
\bigg|
\nonumber\\
&+
\left|
\widetilde{\mathcal{M}}_{\sigma, \mathcal{P}_{\log q^m}}(x)
- 
\widetilde{\mathcal{M}}_{\sigma}(x)
\right|
\nonumber\\
\ll &
\sum_{f\in S_k(q^m)}^{\qquad\prime}
|x| \bigg(\big|\log L_{\mathbb{P}(q)}(\mathrm{Sym}_f^{\mu}, \sigma)- \log L_{\mathcal{P}_{\log q^m}}(\mathrm{Sym}_f^{\mu}, \sigma)-\mathcal{S}_{\mu}\big| +\big|\mathcal{S}_{\mu}\big|
\nonumber\\
&
+
\big|\log L_{\mathbb{P}(q)}(\mathrm{Sym}_f^{\nu}, \sigma)) - \log L_{\mathcal{P}_{\log q^m}}(\mathrm{Sym}_f^{\nu}, \sigma)
-\mathcal{S}_{\nu}\big|+\big|\mathcal{S}_{\nu}\big|\bigg)
\nonumber\\
&+
\bigg|
\sum_{f\in S_k(q^m)}^{\qquad\prime}
\psi_x(\log L_{\mathcal{P}_{\log q^m}}(\mathrm{Sym}_f^{\mu}, \mathrm{Sym}_f^{\nu}, \sigma))
-
\widetilde{\mathcal{M}}_{\sigma, \mathcal{P}_{\log q^m}}(x)
\bigg|
\nonumber\\
&+
\left|
\widetilde{\mathcal{M}}_{\sigma, \mathcal{P}_{\log q^m}}(x)
- 
\widetilde{\mathcal{M}}_{\sigma}(x)
\right|
\nonumber\\
=&
\mathcal{X}_{\log q^m}+\mathcal{Y}_{\log q^m}
+\mathcal{Z}_{\log q^m},
\end{align}
say.
From Proposition~\ref{tildeM}, for any $\varepsilon>0$, there exists
a number $N_0=N_0(\varepsilon)$ 
for which
\[
\left|
\widetilde{\mathcal{M}}_{\sigma, \mathcal{P}_{\log q^m}}(x)
- \widetilde{\mathcal{M}}_{\sigma}(x)
\right|<\varepsilon
\]
holds for any $q^m>N_0$, uniformly in $x\in\mathbb{R}$.
Therefore
\begin{align}\label{lim-Z}
\lim_{q^m\to\infty}\mathcal{Z}_{\log q^m}
=0.
\end{align}

On the estimate of $\mathcal{X}_{\log q^m}$, 
by using \eqref{P1} and \eqref{F},
we find that
\begin{align*}
&
\sum_{f\in S_k(q^m)}^{\qquad\prime}
|x| \bigg(\big|\log L_{\mathbb{P}(q)}(\mathrm{Sym}_f^{\mu}, \sigma)- \log L_{\mathcal{P}_{\log q^m}}(\mathrm{Sym}_f^{\mu}, \sigma)-\mathcal{S}_{\mu}\big|
\nonumber\\
&
+
\big|\log L_{\mathbb{P}(q)}(\mathrm{Sym}_f^{\nu}, \sigma)) - \log L_{\mathcal{P}_{\log q^m}}(\mathrm{Sym}_f^{\nu}, \sigma)-\mathcal{S}_{\nu}\big|
\bigg)
\to 0
\end{align*}
as $q^m$ tends to $\infty$, uniformly in $|x|\leq R$.
Next, by the Cauchy-Schwarz inequality we have
\begin{align*}
&
\sum_{f\in S_k(q^m)}^{\qquad\prime} \big|\mathcal{S}_{\gamma}\big|
\\
\leq &
\bigg(\sum_{f\in S_k(q^m)}^{\qquad\prime} 1^2\bigg)^{1/2}
\bigg(\sum_{f\in S_k(q^m)}^{\qquad\prime} 
\bigg(
\sum_{ p\in \mathbb{P}(q)\setminus \mathcal{P}_{\log q^m}}
\frac{\lambda_f(p^{\gamma})}{ p^{\sigma}}e^{- p/q^{m/4(k-1)\gamma}}
\bigg)^2\bigg)^{1/2}.
\end{align*}
Here, the first factor is $O(1)$ by \eqref{P1}, while the second factor is
\begin{align*}
\ll&
\bigg(\sum_{f\in S_k(q^m)}^{\qquad\prime} 
\sum_{ p\in \mathbb{P}(q)\setminus \mathcal{P}_{\log q^m}}
\frac{e^{- 2p/q^{m/4(k-1)\gamma}}}{p^{2\sigma}}
\\&
+
\sum_{f\in S_k(q^m)}^{\qquad\prime} 
\sum_{\substack{ p,  p' \in \mathbb{P}(q)\setminus \mathcal{P}_{\log q^m}\\  p> p'}}
\frac{\lambda_f(p^{\gamma}{p'}^{\gamma})}{( p p')^{\sigma}}e^{-( p+ p')/q^{m/4(k-1)\gamma}}
\bigg)^{1/2}
\\
\ll&
\bigg(\sum_{f\in S_k(q^m)}^{\qquad\prime} 
\frac{1}{(\log q^m)^{\delta}}
\\&
+
\sum_{\substack{ p,  p' \in \mathbb{P}(q)\setminus \mathcal{P}_{\log q^m}\\  p< p'}}
\frac{(pp')^{(k-1)\gamma/2}E(q^m)}{( p p')^{\sigma}}e^{-( p+ p')/q^{m/4(k-1)\gamma}}
\bigg)^{1/2}.
\end{align*}
Noting
\[
e^{-p/q^{m/4(k-1)\gamma}} 
\ll \bigg(\frac{q^{m/4(k-1)\gamma}}{p}\bigg)^{(k-1)\gamma/2+1/2}
\]
we obtain
\begin{align}\label{X_{log}}
\sum_{f\in S_k(q^m)}^{\qquad\prime} \big|\mathcal{S}_{\gamma}\big|
\ll&
\bigg(\sum_{f\in S_k(q^m)}^{\qquad\prime} 
\frac{1}{(\log q^m)^{\delta}}
\nonumber\\
&+
E(q^m)
\bigg(
\sum_{ p\in \mathbb{P}}
\frac{p^{(k-1)\gamma/2}}{ p^{\sigma}}
\bigg(
\frac{q^{m/4(k-1)\gamma}}{p}
\bigg)^{(k-1)\gamma/2+1/2}
\bigg)^2
\bigg)^{1/2}
\nonumber\\
\ll&
\bigg(
\frac{1}{(\log q^m)^{\delta}}
+
E(q^m)\bigg(\sum_{n=1}^{\infty}\frac{q^{(m/8+m/8(k-1)\gamma)}}{n^{\sigma+1/2}}\bigg)^2
\bigg)^{1/2}
\nonumber\\ 
\ll &
\bigg(
\frac{1}{(\log q^m)^{\delta}}
+
E(q^m)q^{m/2}
\bigg)^{1/2}
\nonumber\\
\ll &
(\log q^m)^{-\delta/2} 
\end{align}
by \eqref{E2}.
Hence we see that
\begin{align}\label{lim-X}
\lim_{q^m\to\infty} \mathcal{X}_{\log q^m}
=
0
\end{align}
uniformly in $|x|\leq R$.

The remaining part of this section is devoted to the estimate of 
$\mathcal{Y}_{\log q^m}$.
We begin with the Taylor expansion
\[
\psi_x(g_{\sigma,  p}(t_p))=\exp(ixg_{\sigma,  p}(t_p))
=1+\sum_{n=1}^{\infty}\frac{(ix)^n}{n!}g_{\sigma,  p}^n(t_p),
\]
where
\begin{align*}
g_{\sigma,  p}^n(t_p)
=&\left(-\log(1-t_p p^{-\sigma})\right)^n
\\
= &
\left(\sum_{j=1}^{\infty}\frac{1}{j}\left(\frac{t_p}{ p^{\sigma}}\right)^j\right)^n
\\
=&
\sum_{a=1}^{\infty} \bigg(\sum_{\substack{a=j_1+\ldots+j_n\\ j_{\ell} \geq 1}}\frac{1}{j_1j_2\cdots j_n}\bigg)
\left(\frac{t_p}{ p^{\sigma}}\right)^a.
\end{align*}
Hence
\begin{align*}
\psi_x(g_{\sigma,  p}(t_p))
= &
1+\sum_{n=1}^{\infty}\frac{(ix)^n}{n!}
\sum_{a=1}^{\infty} \bigg(\sum_{\substack{a=j_1+\ldots+j_n\\ j_{\ell} \geq 1}}\frac{1}{j_1j_2\cdots j_n}\bigg)
\left(\frac{t_p}{ p^{\sigma}}\right)^a\\
=&1+
\sum_{a=1}^{\infty} 
\sum_{n=1}^a\frac{(ix)^n}{n!}\bigg(\sum_{\substack{a=j_1+\ldots+j_n\\ j_{\ell} \geq 1}}\frac{1}{j_1j_2\cdots j_n}\bigg)
\left(\frac{t_p}{ p^{\sigma}}\right)^a,
\end{align*}
which we can write as
\begin{align}\label{G-exp}
\psi_x(g_{\sigma,  p}(t_p))=\sum_{a=0}^{\infty}G_a( p, x) t_p^a
\end{align}
with
\begin{align*}
G_a( p, x)
=& \begin{cases}
       1 & a=0,\\
       \displaystyle{\frac{1}{ p^{a\sigma}}\sum_{n=1}^a\frac{(ix)^n}{n!}\bigg(\sum_{\substack{a=j_1+\ldots+j_n\\ j_{\ell} \geq 1}}\frac{1}{j_1j_2\cdots j_n}\bigg)} 
         & a\geq 1.      
       \end{cases}
\end{align*}
Define
\begin{align*}
G_a(x)
=&
 \begin{cases}
       1 & a=0,\\
       \displaystyle\sum_{n=1}^a\frac{x^n}{n!}
       \left(\begin{matrix}a-1\\n-1\end{matrix}\right)
         & a\geq 1.      
       \end{cases}
\end{align*}
This symbol is the same as (63) 
in Ihara and the first author~\cite{i-1stAuthor}.
Using this symbol we obtain
\begin{equation}\label{G_a}
|G_a( p, x)|
\leq
\frac{1}{ p^{a\sigma}}G_a(|x|)
\end{equation}(see (65) in \cite{i-1stAuthor}).
We use (75), (78) and (79)
in Ihara and the first author~\cite{i-1stAuthor} below.

From \eqref{g_sigma} and \eqref{G-exp} we find that
\begin{align}\label{daiji}
&
\sum_{f\in S_k(q^m)}^{\qquad\prime}
\psi_x(\log L_{\mathcal{P}_{\log q^m}}(\mathrm{Sym}_f^{\mu}, \mathrm{Sym}_f^{\nu}, \sigma))
\nonumber\\
=&
\sum_{f\in S_k(q^m)}^{\qquad\prime}
\prod_{ p \in \mathcal{P}_{\log q^m}}
\psi_x (
g_{\sigma,  p}(\alpha_f^{\mu}( p))+g_{\sigma,  p}(\beta_f^{\mu}( p))
)
\nonumber\\
=&
\sum_{f\in S_k(q^m)}^{\qquad\prime}
\prod_{ p \in \mathcal{P}_{\log q^m}}
\Big(\sum_{a_p=0}^{\infty}G_{a_p}( p, x) 
\alpha_f^{\mu a_p}( p)
\Big)
\Big(\sum_{b_p=0}^{\infty}G_{b_p}( p, x) 
\beta_f^{\mu b_p}( p)
\Big)
\nonumber\\      
=&
\sum_{f\in S_k(q^m)}^{\qquad\prime}
\prod_{ p \in \mathcal{P}_{\log q^m}}
\Big(
\sum_{a_p=0}^{\infty} G_{a_p}^2( p, x) 
\nonumber\\
&
+ 
\sum_{\substack{0\leq a_p, b_p\\ a_p\neq b_p}} 
G_{a_p}( p, x) G_{b_p}( p, x) 
\alpha_f^{\mu a_p}( p) \beta_f^{\mu b_p}( p)
\Big)
\end{align}
and also, from Proposition~\ref{M_P} and \eqref{G-exp}, 
\begin{align}\label{daiji2}
&
\widetilde{\mathcal{M}}_{\sigma, \mathcal{P}_{\log q^m}}(x)
=
\int_{\mathbb{R}}\psi_x(u){\mathcal{M}}_{\sigma, \mathcal{P}_{\log q^m}}(u)\frac{du}{\sqrt{2\pi}}
\nonumber\\
=&
\int_{\mathcal{T}_{\mathcal{P}_{\log q^m}}}\hspace{-2em}
\psi_x(2\Re(g_{\sigma, \mathcal{P}_{\log q^m}}( t_{\mathcal{P}_{\log q^m}})))d^* t_{\mathcal{P}_{\log q^m}}
\nonumber\\
=&
\prod_{ p\in \mathcal{P}_{\log q^m}} 
\int_{\mathcal{T}_p}
\psi_x(2\Re(g_{\sigma,  p}(t_p)))d^*t_p
\nonumber\\
=
&
\prod_{ p\in \mathcal{P}_{\log q^m}} 
\int_{\mathcal{T}_p}
\psi_x(g_{\sigma,  p}(t_p))\psi_x(g_{\sigma,  p}(\overline{t_p}))
d^*t_p
\nonumber\\
=&
\prod_{ p\in \mathcal{P}_{\log q^m}} 
\int_{\mathcal{T}_p}
\Big(\sum_{a_p=0}^{\infty}G_{a_p}( p, x) t_ p^{a_p} \Big)
\Big(\sum_{b_p=0}^{\infty}G_{b_p}( p, x) t_ p^{-b_p}\Big)
d^*t_p
\nonumber\\
=
&
\prod_{ p\in \mathcal{P}_{\log q^m}} 
\Big(\sum_{a_p=0}^{\infty}G_{a_p}^2( p, x) \Big).
\end{align}
Write $\mathcal{P}_{\log q^m}=\{ p_1,\; p_2,\;\ldots,\; p_L\}$,
where $p_l$ means the $l$-th prime number.   (So $L=\pi(\log q^m)$.)
Substituting \eqref{daiji} and \eqref{daiji2} into the definition of
$\mathcal{Y}_{\log q^m}$, and noting \eqref{P1}, we obtain
\begin{align*}
\mathcal{Y}_{\mathcal{P}_{\log q^m}}
= &
\bigg|
E(q^m)\prod_{p\in\mathcal{P}_{\log q^m}}\sum_{a_p=0}^{\infty}G_{a_p}^2(p,x)
\\
&+
\sum_{f\in S_k(q^m)}^{\qquad\prime}
\sum_{\substack{(j_1,\;\ldots,\;j_{L})\neq (0,\ldots,0)\\ j_{\ell}\in \{0, 1\}}}
\prod_{\ell=1}^{L}
\bigg(
\sum_{a_{ p_{\ell}=0}}^{\infty} 
G_{a_{ p_{\ell}}}^2( p_{\ell}, x) \bigg)^{1-j_{\ell}}
\nonumber\\
&\times
\bigg(
\sum_{\substack{0\leq a_{ p_{\ell}}, b_{ p_{\ell}}\\ a_{ p_{\ell}}\neq b_{ p_{\ell}}}} 
G_{a_{ p_{\ell}}}( p_{\ell}, x) G_{b_{ p_{\ell}}}( p_{\ell}, x) 
\alpha_f^{\mu a_{ p_{\ell}}}(p_{\ell})\beta_f^{\mu b_{ p_{\ell}}}(p_{\ell})
\bigg)^{j_{\ell}}\bigg|.
\end{align*}

If we see that
\begin{align}\label{Y_1}
\mathcal{Y}'_{\mathcal{P}_{\log q^m}}
= &
\bigg|
\sum_{f\in S_k(q^m)}^{\qquad\prime}
\sum_{\substack{(j_1,\;\ldots,\;j_{L})\neq (0,\ldots,0)\\ j_{\ell}\in \{0, 1\}}}
\prod_{\ell=1}^{L}
\bigg(
\sum_{a_{ p_{\ell}=0}}^{\infty} 
G_{a_{ p_{\ell}}}^2( p_{\ell}, x) \bigg)^{1-j_{\ell}}
\nonumber\\
&\times
\bigg(
\sum_{\substack{0\leq a_{ p_{\ell}}, b_{ p_{\ell}}\\ a_{ p_{\ell}}\neq b_{ p_{\ell}}}} 
G_{a_{ p_{\ell}}}( p_{\ell}, x) G_{b_{ p_{\ell}}}( p_{\ell}, x) 
\alpha_f^{\mu a_{ p_{\ell}}}(p_{\ell})\beta_f^{\mu b_{ p_{\ell}}}(p_{\ell})
\bigg)^{j_{\ell}}\bigg|
\end{align}
and
\begin{equation}\label{Y''}
\mathcal{Y}''_{\mathcal{P}_{\log q^m}}
=
E(q^m)\prod_{p\in\mathcal{P}_{\log q^m}}\sum_{a_p=0}^{\infty}G_{a_p}^2(p,x)
\end{equation}
tend to $0$, then we obtain that $\mathcal{Y}_{\mathcal{P}_{\log q^m}}$ tends to $0$ 
as $q^m$ tends to $\infty$.

We first consider the second inner sum on the right-hand side of \eqref{Y_1}. 
Letting $a_p-b_p=r_p$ for the part of $a_p > b_p$ and
letting $b_p-a_p=r_p$ for the part of $b_p > a_p$, we obtain
\begin{align*}
&
\sum_{\substack{0\leq a_p, b_p\\ a_p\neq b_p}} 
G_{a_p}( p, x) G_{b_p}( p, x) 
\alpha_f^{\mu a_p}( p)\beta_f^{\mu b_p}( p)
\\
=&
\sum_{0 \leq a_p < b_p} 
+
\sum_{a_p > b_p \geq 0} 
\\
=&
\sum_{r_p \geq 1} 
\sum_{b_p \geq r_p} 
G_{b_p-r_p}( p, x) G_{b_p}( p, x) 
\alpha_f^{\mu (b_p-r_p)}( p)\beta_f^{\mu b_p}( p)
\\
&
+
\sum_{r_p \geq 1}
\sum_{a_p \geq r_p} 
G_{a_p}( p, x) G_{a_p-r_p}( p, x) 
\alpha_f^{\mu a_p}( p)\beta_f^{\mu (a_p-r_p)}( p)
\\
=&
\sum_{r_p \geq 1} 
\sum_{a_p \geq r_p}
G_{a_p}( p, x) G_{a_p-r_p}( p, x) 
\left(
\beta_f^{\mu r_p}( p)
+
\alpha_f^{\mu r_p}( p)
\right)
\\
=&
\sum_{r_p \geq 1} 
\sum_{a_p \geq r_p}
G_{a_p}( p, x) G_{a_p-r_p}( p, x) 
\left(
\lambda_f( p^{\mu r_p})
-
\lambda_f( p^{\mu r_p-2})
\right),
\end{align*}
where the last equation is deduced by 
\begin{align*}
\lambda_f(p^{\mu r_p})
=&
\sum_{h=0}^{\mu r_p}\alpha_f^{\mu r_p -h}(p)\beta_f^h(p) 
\\
=&
\alpha_f^{\mu r_p}+\alpha_f^{\mu r_p -2} +\alpha_f^{\mu r_p -4}+\ldots
+\beta_f^{\mu r_p -4}+\beta_f^{\mu r_p -2} +\beta_f^{\mu r_p }
\end{align*}
which is from \eqref{euler}.
Letting
\[
G_{ p, x}(r)
=\sum_{a_{ p} \geq r}
G_{a_{ p}}( p, x) G_{a_{ p}-r}( p, x),
\]
from \eqref{Y_1} we obtain
\begin{align}\label{daiji3}
\mathcal{Y}'_{\mathcal{P}_{\log q^m}}
=&
\bigg|
\sum_{f\in S_k(q^m)}^{\qquad\prime}
\sum_{\substack{(j_1,\;\ldots,\;j_{L})\neq (0,\ldots,0)\\ j_{\ell}\in \{0, 1\}}}
\prod_{\ell=1}^{L}
\bigg(
\sum_{a_{ p_{\ell}=0}}^{\infty} 
G_{a_{ p_{\ell}}}^2( p_{\ell}, x) \bigg)^{1-j_{\ell}}
\nonumber\\
&\times
\bigg(
\sum_{r_{ p_{\ell}} \geq 1} 
G_{a_{ p_{\ell}},x}(r_{ p_{\ell}})
\big(
\lambda_f( p_{\ell}^{\mu r_{ p_{\ell}}})
-
\lambda_f( p_{\ell}^{\mu r_{ p_{\ell}}-2})
\big)
\bigg)^{j_{\ell}}\bigg|
\nonumber\\
=&
\bigg|
\sum_{\substack{(j_1,\;\ldots,\;j_{L})\neq (0,\ldots,0)\\ j_{\ell}\in \{0, 1\}}}
\prod_{\ell=1}^{L}
\bigg(
\sum_{a_{ p_{\ell}=0}}^{\infty} 
G_{a_{ p_{\ell}}}^2( p_{\ell}, x) \bigg)^{1-j_{\ell}}
\nonumber\\
&\times
\sum_{f\in S_k(q^m)}^{\qquad\prime}
\prod_{\ell=1}^{L}
\bigg(
\sum_{r_{ p_{\ell}} \geq 1} 
G_{a_{ p_{\ell}},x}(r_{ p_{\ell}})
\big(
\lambda_f( p_{\ell}^{\mu r_{ p_{\ell}}})
-
\lambda_f( p_{\ell}^{\mu r_{ p_{\ell}}-2})
\big)
\bigg)^{j_{\ell}}
\bigg|.
\end{align}

\begin{rem}\label{nu}
Here we remark why we only consider the case $\mu-\nu=2$ in the present paper.
If we consider the case that $\nu$ has the same parity with $\mu$
but $\mu-\nu=2h > 2$, and discuss analogously as above, then the factor of the form
\[
\prod_{h=0}^{(\mu-\nu)/2-1}
\prod_{\ell=1}^{L}
\bigg(
\sum_{r_{ p_{\ell}} \geq 1} 
G_{a_{ p_{\ell}},x}(r_{ p_{\ell}})
\big(
\lambda_f( p_{\ell}^{(\mu-2h) r_{ p_{\ell}}})
-
\lambda_f( p_{\ell}^{(\mu-2h) r_{ p_{\ell}}-2})
\big)
\bigg)^{j_{\ell}}
\]
appears.
The summation of this factor over primitive forms cannot be included in the error term, 
because of \eqref{P}.
\end{rem}

Let us continue the argument.
From \eqref{daiji3} we obtain
\begin{align*}
\mathcal{Y}'_{\mathcal{P}_{\log q^m}}
\leq 
\sum_{\substack{(j_1,\;\dots,\;j_{L})\neq (0,\ldots,0)\\ j_{\ell}\in \{0,1\}}}
\prod_{\ell=1}^{L}
|G_{ p_{\ell}, x}(0)|^{1-j_{\ell}}
\Big|
\sum_{f\in S_k(q^m)}^{\qquad\prime}
\sum_{1 < n}\mathcal{G}_x(n)\lambda_f(n)
\Big|,
\end{align*}
where 
\[
\mathcal{G}_x(n)=\begin{cases}
        \displaystyle \prod_{\substack{1\leq\ell\leq L\\j_{\ell}=1}}
(-1)^{r''(p_{\ell})} G_{ p_{\ell},x}(r_{ p_{\ell}})
    &   \displaystyle  n= \prod_{\substack{1\leq\ell\leq L\\j_{\ell}=1}} p_{\ell}^{r'_{ p_{\ell}}}, 
\\
        0 & \mathrm{otherwise},
       \end{cases}
\]
$r'_{ p_{\ell}}= \mu r_{ p_{\ell}}$ or $\mu r_{ p_{\ell}}-2$, and
\[
r''(p_{\ell})=
\begin{cases}
0 & r'_{p_{\ell}}=\mu r_{p_{\ell}}\\ 
1 & r'_{p_{\ell}}=\mu r_{p_{\ell}}-2.
\end{cases}
\] 

We divide the summation of $\mathcal{G}_x(n)\lambda_f(n)$ above
into two parts according to the conditions $n \leq M$ and $n>M$, 
where $M$ is a suitable constant depending on $k$, $\log q^m$ 
and $ p_{\ell}$ $(1\leq \ell \leq L)$ defined below. 
We apply the formula \eqref{P} with $n \neq 1$ to the summation of $n \leq M$. 
And we use $\lambda_f(n)\ll n^{\eta}$ (where $\eta$ is an arbitrarily small positive 
number which will be specified later) by the Ramanujan-Petersson estimate
for the estimation of summation of $n >M$.
We obtain
\begin{align}\label{Y_2}
\mathcal{Y}'_{\mathcal{P}_{\log q^m}}
\leq &
\sum_{\substack{(j_1,\;\dots,\;j_{L})\neq (0,\ldots,0)\\ j_{\ell}\in \{0,1\}}}
\prod_{\ell=1}^{L}
|G_{ p_{\ell}, x}(0)|^{1-j_{\ell}}
\nonumber\\&
\times
\bigg(
E(q^m)
\sum_{1<n \leq M}|\mathcal{G}_x(n)|n^{(k-1)/2}
+
\sum_{n>M} |\mathcal{G}_x(n)|n^{\eta}
\bigg).
\end{align}
From (75), (78) and (79) in \cite{i-1stAuthor} and \eqref{G_a} in this paper, we see
that
\begin{align}\label{coeff1}
|G_{ p_{\ell},x}(0)|
\leq &
\sum_{a_{ p_{\ell}}=0}^{\infty}|G_{a_{ p_{\ell}}}( p_{\ell},x)|^2
\leq 
\sum_{a_{ p_{\ell}}=0}^{\infty}\frac{1}{ p_{\ell}^{2a_{ p_{\ell}}\sigma}}G_{a_{ p_{\ell}}}^2(|x|)
\nonumber\\
\leq &
\left(\sum_{a_{ p_{\ell}}=0}^{\infty}\frac{1}{ p_{\ell}^{a_{ p_{\ell}}\sigma}}G_{a_{ p_{\ell}}}(|x|)\right)^2
=\bigg(\exp\left(\frac{|x|}{ p_{\ell}^{\sigma}-1}\right)\bigg)^2
\end{align}
and
\begin{align}\label{coeff1.5}
|G_{ p_{\ell},x}(r_{ p_{\ell}})|
\leq&
\sum_{a_{ p_{\ell}} \geq r_{ p_{\ell}}}|G_{a_{ p_{\ell}}}( p_{\ell}, x)G_{a_{ p_{\ell}}-r_{ p_{\ell}}}( p_{\ell}, x)|
\nonumber\\
\leq &
\sum_{a_{ p_{\ell}} \geq r_{ p_{\ell}}}\frac{1}{ p_{\ell}^{a_{ p_{\ell}}\sigma} p_{\ell}^{(a_{ p_{\ell}}-r_{ p_{\ell}})\sigma}}
G_{a_{ p_{\ell}}}(|x|)G_{a_{ p_{\ell}}-r_{ p_{\ell}}}(|x|)
\nonumber\\
\leq &
\sum_{a'_{ p_{\ell}}=0}^{\infty}\frac{1}{ p_{\ell}^{(a'_{ p_{\ell}}+r_{ p_{\ell}})\sigma} p_{\ell}^{a'_{ p_{\ell}}\sigma}}
G_{a'_{ p_{\ell}}+r_{ p_{\ell}}}(|x|)G_{a'_{ p_{\ell}}}(|x|)
\nonumber\\
\leq &
\frac{1}{ p^{r_{ p_{\ell}}\sigma}}
\sum_{a'_{ p_{\ell}}=0}^{\infty}\frac{1}{ p_{\ell}^{a'_{ p_{\ell}}\sigma} p_{\ell}^{a'_{ p_{\ell}}\sigma}}
G_{a'_{ p_{\ell}}}(|x|)G_{a'_{ p_{\ell}}}(|x|)L_{r_{ p_{\ell}}}(|x|)
\nonumber\\
\leq &
\frac{L_{r_{ p_{\ell}}}(|x|)}{ p_{\ell}^{r_{ p_{\ell}}\sigma}}
\bigg(\exp\left(\frac{|x|}{ p_{\ell}^{\sigma}-1}\right)\bigg)^2
\nonumber\\
\leq &
\frac{1}{ p^{r_{ p_{\ell}}\sigma/2}}
\bigg(\exp\left(\frac{|x|}{ p_{\ell}^{\sigma}-1}\right)\bigg)^2
\bigg(\exp\left(\frac{|x|}{ p_{\ell}^{\sigma/2}-1}\right)\bigg),
\end{align}
where
\[
L_r(x)=\sum_{m=0}^r G_m(x)
\]
(the same as [74] in \cite{i-1stAuthor}).
Therefore, when
\[
n=
\prod_{\substack{1\leq\ell\leq L\\ j_{\ell}=1}}
p_{\ell}^{r'_{ p_{\ell}}}
\]
with $r'_{p_{\ell}}=\mu r_{p_{\ell}}$ or $\mu r_{p_{\ell}}-1$ for $r_{p_{\ell}}\geq 1$,
from \eqref{coeff1.5} we have
\begin{align}\label{cal-G}
\mathcal{G}_x(n)
\leq &
\prod_{\substack{1\leq\ell\leq L\\ j_{\ell}=1}}
|G_{ p_{\ell}, x}(r_{ p_{\ell}})|
\nonumber\\
\leq &
\prod_{\substack{1\leq\ell\leq L\\ j_{\ell}=1}}
\frac{1}{ p_{\ell}^{r_{ p_{\ell}}\sigma/2}}
\left(\exp\left(\frac{|x|}{ p_{\ell}^{\sigma}-1}\right)\right)^2
\bigg(\exp\bigg(\frac{|x|}{ p_{\ell}^{\sigma/2}-1}\bigg)\bigg)
\nonumber\\
\leq &
\frac{1}{n^{\sigma/2\mu}}
\prod_{\substack{1\leq\ell\leq L\\ j_{\ell}=1}}
\left(\exp\left(\frac{|x|}{ p_{\ell}^{\sigma}-1}\right)\right)^2
\bigg(\exp\bigg(\frac{|x|}{ p_{\ell}^{\sigma/2}-1}\bigg)\bigg).
\end{align}
From \eqref{Y_2}, \eqref{coeff1} and \eqref{cal-G},
we obtain
\begin{align*}
\mathcal{Y}'_{\mathcal{P}_{\log q^m}}
\ll&
\sum_{\substack{(j_1,\;\dots,\;j_{L})\neq  (0,\ldots,0)\\ j_{\mu}\in \{0,1\}}}
\prod_{\ell=1}^{L}
\left(\exp\left(\frac{|x|}{ p_{\ell}^{\sigma}-1}\right)\right)^2
\bigg(\exp\bigg(\frac{|x|}{ p_{\ell}^{\sigma/2}-1}\bigg)\bigg)^{j_{\ell}}
\\
&
\times
\bigg(
E(q^m)
\sum_{\substack{n \leq M \\  p \nmid n\;\mathrm{for}\;  p>\log q^m}} 
\frac{n^{(k-1)/2}}{n^{\sigma/2\mu}}
+
\sum_{\substack{n > M \\  p \nmid n\;\mathrm{for}\;  p>\log q^m}} 
\frac{n^{\eta}}{n^{\sigma/2\mu}}
\bigg).
\end{align*}
Here we choose $\eta=\sigma/4\mu$.   Then the second inner sum is
\begin{align*}
&
\sum_{\substack{n > M \\  p \nmid n\;\mathrm{for}\;  p>\log q^m}} 
\frac{1}{n^{\sigma/4\mu}}
<
\frac{1}{M^{\sigma/8\mu}}
\sum_{\substack{n > M \\  p \nmid n\;\mathrm{for}\;  p>\log q^m}} 
\frac{1}{n^{\sigma/8\mu}}
\\
< &
\frac{1}{M^{\sigma/8\mu}}
\prod_{l=1}^{L}\frac{1}{1-p_l^{-\sigma/8\mu}}
=
\frac{1}{M^{\sigma/8\mu}}
\prod_{l=1}^{L}\frac{p_l^{\sigma/8\mu}}{p_l^{\sigma/8\mu}- 1}
\\
\ll &
\frac{1}{M^{\sigma/8\mu}}
\prod_{l=1}^{L}\frac{p_l^{\sigma/8\mu}}{p_l^{\sigma/8\mu}- p_l^{\sigma/8\mu}/2}
\ll 
\frac{2^{L}}{M^{\sigma/8\mu}}.
\end{align*}
Hence we obtain
\begin{align*}
\mathcal{Y}'_{\mathcal{P}_{\log q^m}}
&\leq 
\sum_{\substack{(j_1,\;\dots,\;j_{L})\neq (0,\ldots,0)\\ j_{\ell}\in \{0, 1\}}}
\prod_{\ell=1}^{L}
\left(\exp\left(\frac{|x|}{ p_{\ell}^{\sigma}-1}\right)\right)^2
\bigg(\exp\bigg(\frac{|x|}{ p_{\ell}^{\sigma/2}-1}\bigg)\bigg)^{j_{\ell}}
\\
&\quad
\times
\bigg(
E(q^m)
M^{(k+1)/2-\sigma/2\mu}
+
\frac{2^{L}}{M^{\sigma/8\mu}}
\bigg)
\\
&\leq 
2^{L}
\prod_{\ell=1}^{L}
\left(\exp\left(\frac{|x|}{ p_{\ell}^{1/4}-1}\right)\right)^3
\\
&\quad
\times
\bigg(
E(q^m)M^{(k+1)/2-1/4\mu}
+
\frac{2^{L}}{M^{1/16\mu}}
\bigg),
\end{align*}
since $\sigma>1/2$.
For large $\ell$, we know
\begin{align}\label{boundby2}
\exp\bigg(\frac{|x|}{ p_{\ell}^{1/4}-1}\bigg)
\leq
\exp\bigg(\frac{R}{ p_{\ell}^{1/4}-1}\bigg)
<2
\end{align}
for $|x|\leq R$, hence
\begin{align*}
\prod_{\ell=1}^{L}
\left(\exp\left(\frac{|x|}{ p_{\ell}^{1/4}-1}\right)\right)^3
 \ll_{R} 
2^{3L}.
\end{align*}
Therefore we obtain
\begin{equation}\label{Y_*}
\mathcal{Y}'_{\mathcal{P}_{\log q^m}}
\leq 
2^{4L}
\bigg(
E(q^m)M^{c(k,\mu)}
+
\frac{2^{L}}{M^{1/16\mu}}
\bigg),
\end{equation}
where $c(k,\mu)=(k+1)/2-1/4\mu$.
Noting the fact that the number of the prime numbers less than 
$2^5=32$ is $11$, we choose
\[
M=
\left(
\frac{ p_1\cdots  p_{L}}{(2^5)^{L-11}}
\right)^{1/c(k,\mu)}.
\]
Since $c(k,\mu)\leq 15/2<8$, we see that
\begin{align*}
M=&
( p_1\cdots  p_{11})^{1/c(k,\mu)} 
\left(\frac{ p_{12}}{2^5}\cdot \frac{ p_{13}}{2^5}\cdot
\frac{ p_{L}}{2^5}\right)^{1/c(k,\mu)} 
\\
>&
( p_1\cdots  p_{11})^{2/15}
>
(200560490130)^{1/8}
>
25.
\end{align*}
For large $m$ or $q$, by the prime number theorem we have
\begin{align*}
p_1\cdots p_L
&=\exp(\log p_1+\cdots+\log p_L)\\
&=\exp(\log q^m+O((\log q^m)\exp(-c_1\sqrt{\log\log q^m}))\\
&\leq q^{m(1+c_2\exp(-c_1\sqrt{\log\log q^m}))}
\end{align*}
where $c_1$, $c_2$ are positive constants.
By using \eqref{E2}, we have
\begin{align}\label{Y1}
&
2^{4L}E(q^m)M^{c(k,\mu)}
\leq 
2^{4L}E(q^m)
\cdot\frac{ p_1\cdots  p_{L}}{2^{5L-55}}
\nonumber\\
\ll &
\frac{ p_1\cdots  p_{L}}{2^{L}}\cdot E(q^m)
\ll 
\frac{q^{m(1+c_2\exp(-c_1\sqrt{\log\log q^m}))}}{2^{L}}
\cdot
q^{-m}.
\end{align}
Again by the prime number theorem, we see that
\[
2^L=2^{\log q^m(1+o(1))/\log\log q^m}=(q^m)^{\log 2(1+o(1))/\log\log q^m},
\]
so we find that the right-hand side of \eqref{Y1} is
\[
\ll (q^m)^{c_2\exp(-c_1\sqrt{\log\log q^m})-\log 2(1+o(1))/\log\log q^m},
\]
whose exponent is negative for large $q^m$.
Therefore this tends to 0 as $q^m$ tends to $\infty$.

Next, we have
\begin{align*}
\frac{2^{4L}2^{L}}{M^{1/16\mu}}
=&
2^{5L}
\left(
\frac{ p_1\cdots  p_{L}}{2^{5L-55}}
\right)^{-1/16\mu c(k,\mu)} 
\nonumber\\
\ll &
\frac{(2^{L})^{5+5/16\mu c(k,\mu)}}
{( p_1\cdots  p_{L})^{1/16\mu c(k,\mu)}} 
\nonumber\\
\ll &
\left(\frac{(2^{80\mu c(k,\mu)+5})^{L}}
{ p_1\cdots  p_{L}}\right)^{1/16\mu c(k,\mu)},
\end{align*}
so, putting $d(k,\mu)=2^{80\mu c(k,\mu)+5}$, the above is
\begin{align}\label{Y2}
=&
\left(
\frac{d(k,\mu)}{p_1}\cdots\frac{d(k,\mu)}{ p_{\pi(d(k,\mu))}}
\frac{(d(k,\mu))^{L-\pi(d(k,\mu))}}{ p_{\pi(d(k,\mu))+1}\cdots  p_{L}}
\right)^{1/16\mu c(k,\mu)} 
\nonumber\\
&\ll_k
\left(
\frac{d(k,\mu)}{ p_{\pi(d(k,\mu))+1}}
\right)^{(L-\pi(d(k,\mu))/16\mu c(k,\mu)}.
\end{align}
Since the quantity in the parentheses is smaller than 1, we find that
this also tends to 0 as $q^m$ tends to $\infty$.
Therefore from \eqref{Y_*} we conclude that
$\mathcal{Y}'_{\mathcal{P}_{\log q^m}}$ tends to 0 as $q^m$ tends to $\infty$.

The idea of evaluating $\mathcal{Y}''_{\mathcal{P}_{\log q^m}}$,
defined by \eqref{Y''}, is
essentially similar, but much simpler.  
First, using \eqref{coeff1}, we have
\begin{align*}
\mathcal{Y}''_{\mathcal{P}_{\log q^m}}
=&
E(q^m)\prod_{p\in\mathcal{P}_{\log q^m}}\sum_{a_p=0}^{\infty}G_{a_p}^2(p,x)
\leq
E(q^m)\prod_{p\in\mathcal{P}_{\log q^m}}|G_{p,x}(0)|
\nonumber\\
\leq &
E(q^m)\prod_{p\in\mathcal{P}_{\log q^m}}\bigg(\exp\bigg(\frac{|x|}{p^{\sigma}-1}\bigg)\bigg)^2.
\end{align*}
Then, by an argument similar to \eqref{boundby2}, the above is
\begin{align*}
\ll_R
E(q^m)2^{L}
\ll
E(q^m)e^{\log q^m/\log\log q^m} 
=E(q^m)(q^m)^{(1+o(1))/\log\log q^m},
\end{align*}
which tends to 0
as $q^m$ tends to $\infty$.
Therefore we now arrive at the assertion
\begin{align}\label{lim-Y}
\lim_{q^m\to\infty}\mathcal{Y}_{\mathcal{P}_{\log q^m}} 
=0.
\end{align}
Finally we see that Lemma~\ref{keylemma} is established,
by substituting \eqref{lim-Z}, \eqref{lim-X} and \eqref{lim-Y} into \eqref{basic}.

\section{Completion of the proof of Theorem~\ref{main}.}

The only remaining task now is to deduce the general statement of our
Theorem~\ref{main} from Lemma~\ref{keylemma}.
This can be done by using the general principle on the weak convergence of probability
measures (as indicated in Remark 3.2 of \cite{i-m3}), but here we follow a more
self-contained treatment given in Ihara and the first author~\cite{i-1stAuthor}.
In this section, we just explain the outline of the proof of Theorem~\ref{main},
because the argument is the same as that in \cite{i-1stAuthor}.

For any $\varepsilon >0$, the aim of this section
is to prove that 
\begin{equation}\label{A1}
\sum_{f\in S(q^m)}^{\qquad\prime} \Psi\circ
\log L_{\mathbb{P}(q)}(\mathrm{Sym}_f^{\mu}, \mathrm{Sym}_f^{\nu}, \sigma)
\end{equation}
tends to
\[
\int_{\mathbb{R}}
\mathcal{M}_{\sigma}(u)
\Psi(u)\frac{du}{\sqrt{2\pi}}
\]
as $q^m$ tends to $\infty$.

We define the set $\Lambda$ of the function $\Phi$ on $\mathbb{C}$ by
\[
\Lambda
=
\{\Phi \in L^1\cap L^{\infty} \mid \Phi^{\vee} \in L^1 \cap L^{\infty}, 
(\Phi^{\vee})^{\wedge}=\Phi\},
\]
where $\Phi^{\wedge}$ means the Fourier transform of $\Phi$
and
$\Phi^{\vee}$ means the Fourier inverse transform of $\Phi$.
We know
\[
\Phi(u)
=\int_{\mathbb{R}}\Phi^{\wedge}(x)\psi_{-u}(x)\frac{dx}{\sqrt{2\pi}}
=\int_{\mathbb{R}}\Phi^{\wedge}(x)\psi_{-x}(u)\frac{dx}{\sqrt{2\pi}}.
\]
Since we
also know $\widetilde{\mathcal{M}}_{\sigma} \in \Lambda$ from
Proposition~\ref{tildeM} and Proposition~\ref{M}, 
we have $\mathcal{M}_{\sigma} \in \Lambda$.
Therefore we have
\begin{align*}
\int_{\mathbb{R}}
\mathcal{M}_{\sigma}(u)
\Phi(u)\frac{du}{\sqrt{2\pi}}
=&
\int_{\mathbb{R}}
\overline{
\mathcal{M}_{\sigma}}^{\wedge}(x)\Phi^{\wedge}(x)
\frac{dx}{\sqrt{2\pi}}
\\
=&
\int_{\mathbb{R}}
\widetilde{{\mathcal{M}}_{\sigma}}(-x)
\Phi^{\wedge}(x)
\frac{dx}{\sqrt{2\pi}}.
\end{align*}

%
%

In the case of $\Psi=\Phi\in\Lambda$,
we can see
\begin{align*}
&
\bigg|
\sum_{f\in S(q^m)}^{\qquad\prime}
\Phi \circ \log L_{\mathbb{P}(q)}(\mathrm{Sym}_f^{\mu}, \mathrm{Sym}_f^{\nu}, \sigma)
-
\int_{\mathbb{R}}
\mathcal{M}_{\sigma}(u)
\Phi(u)\frac{du}{\sqrt{2\pi}}
\bigg| 
\\
= &
\bigg|
\sum_{f\in S_k(q^m)}^{\qquad\prime}
\int_{\mathbb{R}}
(\Phi^{\wedge}(x)
\psi_{-x}(
\log L_{\mathbb{P}(q)}(\mathrm{Sym}_f^{a}, \mathrm{Sym}_f^{\nu}, \sigma)
))\frac{dx}{\sqrt{2\pi}}
\\&
-
\int_{\mathbb{R}}
\widetilde{\mathcal{M}_{\sigma}}(-x)
\Phi^{\wedge}(x)\frac{dx}{\sqrt{2\pi}}
\bigg| 
\\
= &
\bigg|
\sum_{f\in S_k(q^m)}^{\qquad\prime}
\int_{\mathbb{R}}
(\Phi^{\wedge}(-x)
\psi_x(
\log L_{\mathbb{P}(q)}(\mathrm{Sym}_f^{\mu}, \mathrm{Sym}_f^{\nu}, \sigma)
))\frac{dx}{\sqrt{2\pi}}
\\&
-
\int_{\mathbb{R}}
\widetilde{\mathcal{M}_{\sigma}}(x)
\Phi^{\wedge}(-x)\frac{dx}{\sqrt{2\pi}}
\bigg| 
\\
\leq &
\int_{\mathbb{R}}
|\Phi^{\wedge}(-x)| \\
&\times
\bigg|
\sum_{f\in S_k(q^m)}^{\qquad\prime}
(
\psi_x(
\log L_{\mathbb{P}(q)}(\mathrm{Sym}_f^{\mu}, \mathrm{Sym}_f^{\nu}, \sigma)
))
-
\widetilde{\mathcal{M}}_{\sigma}(x)
\bigg|
\frac{dx}{\sqrt{2\pi}}
.
\end{align*}
We divide the above integral into two parts, $|x|\leq R$ and $|x|>R$
by sufficiently large $R$.
Since $\psi_x$ and $\widetilde{\mathcal{M}}_{\sigma}(x)$ are bounded (see
\eqref{tildeTrivial}) and $\Phi^{\wedge}\in L^1$, 
the integral on $|x|>R$ is small for sufficiently large $R$.
The other integral on $|x|\leq R$ is then also small by Lemma~\ref{keylemma},
for large $q$ or $m$.    Therefore the desired assertion holds for
$\Psi\in\Lambda$.

In the case that $\Psi$ is a compactly supported function on $C^{\infty}$,
this is a element in the Schwartz space. Therefore $\Psi\in \Lambda$.

In the case that $\Psi$ is a compactly supported continuous function,
this is approximated by compactly supported functions in $C^{\infty}$.
Therefore, in this case Theorem~\ref{main} is established. Especially,
in the case that $\Psi$ is a characteristic function on a compact subset, 
$\Psi$ is approximated by compactly supported continuous function.
Therefore, in this case the proof is complete.

Finally, we consider
the case that $\Psi$ is a bounded continuous function.
For any $R>0$, 
there exists a compactly supported continuous function $\Psi_R$
such that
$\Psi_R(x)=\Psi(x)$ for $|x|\leq R$ and 
$|\Psi_R(x)|\leq |\Psi(x)|$ for $|x|>R$.
We already know that 
\[
\lim_{\substack{q\to \infty\\ \mathrm{or}\; m\to\infty}} 
\sum_{f\in S(q^m)}^{\qquad\prime} \Psi_R \circ 
\log L_{\mathbb{P}(q)}(\mathrm{Sym}_f^{\mu}, \mathrm{Sym}_f^{\nu}, \sigma)
=
\int_{\mathbb{R}}
\mathcal{M}_{\sigma}(u)
\Psi_R(u)\frac{du}{\sqrt{2\pi}},
\]
where 
the above equation is proved for $q\geq Q(\mu)$ when $1 \geq \sigma >1/2$
in the case of $m \to \infty$.
For the right-hand side of this equation, we have
\begin{align*}
&
\int_{\mathbb{R}}
\mathcal{M}_{\sigma}(u)
\Psi_R(u)\frac{du}{\sqrt{2\pi}}
\\
=&
\int_{|x|> R}
\mathcal{M}_{\sigma}(u)
\Psi_R(u)\frac{du}{\sqrt{2\pi}}
+
\int_{|x| \leq R}
\mathcal{M}_{\sigma}(u)
\Psi_R(u)\frac{du}{\sqrt{2\pi}}
\\
=&
\int_{|x|> R}
\mathcal{M}_{\sigma}(u)
\Psi_R(u)\frac{du}{\sqrt{2\pi}}
+
\int_{|x| \leq R}
\mathcal{M}_{\sigma}(u)
\Psi(u)\frac{du}{\sqrt{2\pi}}
\\
=&
\int_{|x|> R}
\mathcal{M}_{\sigma}(u)
(\Psi_R(u)-\Psi(u))\frac{du}{\sqrt{2\pi}}
+
\int_{\mathbb{R}}
\mathcal{M}_{\sigma}(u)
\Psi(u)\frac{du}{\sqrt{2\pi}}.
\end{align*}
As for the former integral, we remind that
${\mathcal{M}}_{\sigma}$ is non-negative to find
\[
\int_{|x|> R}
\mathcal{M}_{\sigma}(u)
(\Psi_R(u)-\Psi(u))\frac{du}{\sqrt{2\pi}}
\ll
\int_{|x|> R}
\mathcal{M}_{\sigma}(u)
\frac{du}{\sqrt{2\pi}}
\]
which tends to $0$ as $R$ tends to $\infty$,
since we know
\[
\int_{\mathbb{R}}
\mathcal{M}_{\sigma}(u)
\frac{du}{\sqrt{2\pi}}
= 1
\]
by
Proposition~\ref{M}.
Hence we have
\begin{align}\label{finalstage}
&
\lim_{R\to\infty}
\lim_{\substack{q\to\infty\\\mathrm{or}\;m\to\infty}} 
\sum_{f\in S(q^m)}^{\qquad\prime} \Psi_R \circ 
\log L_{\mathbb{P}(q)}(\mathrm{Sym}_f^{\mu}, \mathrm{Sym}_f^{\nu}, \sigma)
\nonumber\\
=&
\lim_{R\to\infty}
\int_{\mathbb{R}}
\mathcal{M}_{\sigma}(u)
\Psi_R(u)\frac{du}{\sqrt{2\pi}}
=
\int_{\mathbb{R}}
\mathcal{M}_{\sigma}(u)
\Psi(u)\frac{du}{\sqrt{2\pi}}.
\end{align}
Finally, by using the same argument as in p.675 in Ihara and the first 
author~\cite{i-1stAuthor}, from \eqref{finalstage} we obtain
\[
\lim_{\substack{q\to\infty\\\mathrm{or}\; m\to\infty}}
\sum_{f\in S(q^m)}^{\qquad\prime} \Psi \circ \log L_{\mathbb{P}(q)}(\mathrm{Sym}_f^{\mu}, \mathrm{Sym}_f^{\nu}, \sigma) 
=
\int_{\mathbb{R}}
\mathcal{M}_{\sigma}(u)
\Psi(u)\frac{du}{\sqrt{2\pi}},
\]
which is the conclusion of our Theorem~\ref{main}.

Kohji Matsumoto:\\
Graduate School of Mathematics,
Nagoya University, Furocho, Chikusa-ku, Nagoya 464-8602, Japan.\\
kohjimat@math.nagoya-u.ac.jp
\bigskip

\noindent
Yumiko Umegaki:\\
Department of Mathematical and Physical Sciences,
Nara Women's University,
Kitauoya Nishimachi, Nara 630-8506, Japan.\\
ichihara@cc.nara-wu.ac.jp
\end{document}